\documentclass[11pt, reqno]{amsart}

\DeclareMathAccent{\mathring}{\mathalpha}{operators}{"17}

\newcommand{\mysection}[1]{\section{#1}
      \setcounter{equation}{0}}

\newcommand{\nlimsup}{\operatornamewithlimits{\overline{lim}}}
\newcommand{\nliminf}{\operatornamewithlimits{\underline{lim}}}

\newtheorem{theorem}{Theorem}[section]
\newtheorem{lemma}[theorem]{Lemma}

\newtheorem{corollary}[theorem]{Corollary}

\theoremstyle{definition}
\newtheorem{assumption}{Assumption}[section]

\theoremstyle{remark}
\newtheorem{remark}{Remark}[section]

\newcommand\cbrk{\text{$]$\kern-.15em$]$}}
\newcommand\opar{\text{\raise.2ex\hbox{${\scriptstyle | }$}\kern-.34em$($} }

 \makeatletter
 \def\dashint{%
 \operatorname%
 {\,\,\text{\bf--}\kern-.98em\DOTSI\intop\ilimits@\!\!}}
 \makeatother

\newcommand{\WO}{\mathring{W}}

\newcommand\bR{\mathbb{R}}

\newcommand\bC{\mathbb{C}}

\newcommand\bL{\mathbb{L}}
\def\bM{\mathbb{M}}

\newcommand\bZ{\mathbb{Z}}

\newcommand\cF{\mathcal{F}}

\newcommand\cQ{\mathcal{Q}}
\newcommand\cD{\mathcal{D}}

\newcommand\cS{\mathcal{S}}

\newcommand\cU{\mathcal{U}}

\newcommand\frC{\mathfrak{C}}

\begin{document}

\title[Fully nonlinear equations with VMO coefficients
]{On fully nonlinear elliptic and parabolic equations in
domains with VMO coefficients }

\author{Hongjie Dong}
\thanks{The first author was partially supported by
NSF grant DMS-0800129}
\email{Hongjie\_Dong@brown.edu}
\address{182 George Street, Brown University, Providence, RI, 02912.}

\author{N.V. Krylov}
\thanks{The second author was partially supported by
NSF Grant DMS-0653121}
\email{krylov@math.umn.edu}
\address{127 Vincent Hall, University of Minnesota,
 Minneapolis, MN, 55455}

\author{Xu Li}
\email{lixxx489@umn.edu}
\address{127 Vincent Hall, University of Minnesota,
 Minneapolis, MN, 55455}



\begin{abstract}
We prove  the  solvability in Sobolev spaces $W^{1,2}_p$,
$p > d+1$, of the terminal-boundary value problem for
a class of fully nonlinear parabolic equations, including parabolic
Bellman's equations,  in bounded cylindrical
domains with VMO ``coefficients''. The solvability in
$W^{2}_p$, $p > d$, of the corresponding
 elliptic boundary-value problem is also obtained.
\end{abstract}

\maketitle

\mysection{Introduction and main results}

In this article, we consider parabolic  equations
$$
\partial_t u(t,x) + F(D^{2}u(t,x),t,x)
$$
\begin{equation}
                                                \label{7.29.1}
+G(D^{2}u(t,x),D u(t,x),u(t,x),t,x)=0
\end{equation}
in subdomains of $\bR^{d+1}=\{(t,x):t\in\bR,x\in\bR^{d}\}$, where
$$
\bR^{d}=\{x=(x^{1},...,x^{d}):x^{1},...,x^{d}\in
\bR=(-\infty,\infty)\}.
$$
Here
$$
D^{2}u=(D_{ij}u),\quad Du=(D_{i}u),\quad
D_{i}=\frac{\partial}{\partial x^{i}},\quad D_{ij}=D_{i}D_{j},\quad
\partial_{t}=\frac{\partial}{\partial t}.
$$
We introduce $\cS$ as the set of symmetric $d\times d$ matrices,
fix some constants
$\delta\in(0,1)$ and  $K\in\bR_{+}:=
(0,\infty)$,  and throughout the article we assume
that

$({\rm H}_{1})$
$F(u'',t,x)$ is convex and positive homogeneous of degree one
with respect to $u''\in\cS$ and
 for all values
of the arguments  and $\xi\in\bR^{d}$
$$
 \delta|\xi|^2 \leq F(u''+
\xi\xi^{*} ,t,x)
-F(u'',t,x) \leq \delta^{-1}|\xi|^2;
$$

$({\rm H}_{2})$  $G(u'',u',u,t,x)$, $u''\in\cS,u'\in\bR^{d},u\in\bR$, is nonincreasing
in $u$ and  for all values
of the arguments (notice $u''$ and not $v''$)
$$
|G(u'',u',u,t,x)-
G(u'',v',v,t,x)|\leq K\big(|u'-v'|
+|u-v| \big),
$$
$$
|G(u'',u',u,t,x)|\leq
\chi(|u''|)|u''|+ K(|u'|+|u|)+\bar{G}(t,x),
$$
where $\bar{G}$ and $\chi$ are given functions such that
$\chi:\bar\bR_+\to \bar \bR_+$ is bounded, decreasing, and
$\chi(r)\to0$ as
$r\to\infty$;

$({\rm H}_{3})$ $F(u'',t,x)+G(u'',u',u,t,x)$ is convex with respect
to $u''\in \cS$ and for all values of the arguments and
$\xi\in\bR^{d}$
$$
\delta|\xi|^2 \leq F(u''+\xi\xi^{*} ,t,x)
+G(u''+\xi\xi^{*} ,u',u,t,x)
$$
$$
-F(u'',t,x) -G(u''  ,u',u ,t,x)\leq \delta^{-1}|\xi|^2 .
$$

We shall derive a priori estimates only using conditions
 $({\rm H}_{1})$ and $({\rm H}_{2})$. It is in the proofs of the
solvability where condition $({\rm H}_{3})$ plays its role.
However, we have the following

{\em Conjecture}. In (H$_{3}$) the convexity assumption on
$F(u'',t,x)+G(u'',u',u,t,x)$ with respect
to $u''\in \cS$
can be dropped.

To state our main results, we introduce a few notation.
For $r>0$,
 $x\in \bR^d$, and $t\in \bR$, we denote
$$
B_r(x)=\{y\in\bR^{d}: |x-y|<r\},\quad
Q_r(t,x)=(t,t+r^2)\times B_r(x).
$$
If $\cD$ is a domain in $\bR^d$ and $-\infty
\le S<T<\infty$, we denote the parabolic boundary of the cylinder
$(S,T)\times \cD$ by
$$
\partial' ((S,T)\times \cD)=(\{T\}\times \cD)\cup ((S,T]\times \partial \cD).
$$
Finally, for any $T>0 $, we define
$\cD_T=(0,T)\times \cD$.

The following
 VMO (vanishing mean oscillation)  assumption is
imposed on the leading term in \eqref{7.29.1}
with a constant $\theta\in (0,1]$ to be specified later.

\begin{assumption}
                                \label{assump1}
There exists $R_0\in(0,1]$ such that for any
$r\in (0, R_0]$, $\tau\in \bR$, and  $z\in \cD$
one can find a function $\bar{F} (u'' )$ (independent
of $(t,x)$)
satisfying condition $({\rm H}_{1})$ and such that
for any $u''\in\cS$ with $|u''|=1$ we have
\begin{equation}
                                                     \label{7.30.2}
\int_{Q_{r}(\tau,z)}
|
F(u'' ,t,x)-\bar{F}(u'')| \,dx\,dt\leq \theta
r^{ d+2}.
\end{equation}
\end{assumption}

The first main result of the article is about
the terminal-boundary value problem for  fully nonlinear
parabolic equations with ``VMO
coefficients''
 in bounded cylinders.

\begin{theorem}
                                    \label{mainthm1}
Let $p>d+1$ be a constant,
$T\in\bR_{+}$, and let $\cD$ be a bounded
$C^{1,1}$ domain in
$\bR^d$.
Assume that $\bar G\in L_{p}(\cD_{T})$.
Then there exists a constant $\theta\in (0,1]$ depending
only on $d$, $p$, $\delta$, and the $C^{1,1}$ norm of $\partial \cD$
such that if Assumption \ref{assump1} is satisfied with this
 $\theta$, then the following assertions hold.
For any $g\in W^{1,2}_p(\cD_T)$, there is a unique solution
$u\in W^{1,2}_p(\cD_T)$ to \eqref{7.29.1} such
that $u-g\in \WO^{1,2}_p(\cD_T)$. Moreover, we have
\begin{equation}
                                \label{eq16.01}
\|u\|_{W^{1,2}_p(\cD_T)}\le N
\| \bar G\|_{W^{1,2}_p(\cD_T)}+N\|g\|_{W^{1,2}_p(\cD_T)}
+ N_{0} ,
\end{equation}
where $N$ depends only on $d$, $p$, $\delta$, $K$,
$R_{0}$, the  $C^{1,1}$
norm of $\partial \cD$, and $\text{\rm diam}(\cD)$
 and $ N_{0} $ depends only on the same objects,  $T$,  and
$\chi$. In particular, $ N_{0} =0$ if $\chi\equiv0$.
\end{theorem}

Here $W^{1,2}_p(\cD_T)$ denotes the set of
functions $v$ defined on
$\cD_T$ such that $v$, $Dv
 $, $D^2v $, and $\partial_t v$
are in
$L_p(\cD_T)$, and $\WO^{1,2}_p(\cD_T)$ is the set of all functions
$v\in
W^{1,2}_p(\cD_T)$ such that $v$ vanishes on $\partial' \cD_T$.

If $F$ and $G$ are independent of $t$, we also
   consider elliptic
    equations
\begin{equation}
                                                \label{7.29.2}
 F(D^{2}u( x), x)+G(D^{2}u( x),D u( x),u( x), x)=0
\end{equation}
in subdomains of $\bR^{d}$ with Dirichlet boundary condition.
In that case Assumption \ref{assump1} becomes the
following.

\begin{assumption}
                                \label{assump2}
There exists $R_0\in(0,1]$ such that, for any
 $r\in (0, R_0]$ and $z\in \cD$,
one can find a function $\bar{F}(u'')$ (independent of $x$)
satisfying condition $({\rm H}_{1})$ and such that
for any $u''\in\cS$ with $|u''|=1$ we have
\begin{equation}
                                            \label{7.28.1}
\int_{B_{r}( z)}
|F(u'' , x)-\bar{F}(u'' )|\, dx \leq \theta
r^{ d} .
\end{equation}

\end{assumption}

Our next theorem is about the boundary value problem
for elliptic  equations with VMO coefficients
in bounded domains.

\begin{theorem}
                                    \label{mainthm2}
Let $p>d$ be a constant and
$\cD$ be a bounded  $C^{1,1}$  domain in $\bR^d$.
Assume that $\bar G$ is independent of $t$
and $\bar G\in L_{p}(\cD)$.
Then there exists a constant $\theta\in (0,1]$,
depending only on $d$, $p$, $\delta$, and the  $C^{1,1}$
norm of $\partial \cD$, such that if Assumption \ref{assump2}
 is satisfied with this $\theta$, then the following
assertions hold. For any $g\in W^{ 2}_p(\cD)$  there is a unique solution
$u\in W^{2}_p(\cD)$ to \eqref{7.29.2} such that
$u-g\in \WO^{2}_p(\cD)$. Moreover, we have
\begin{equation}
                                \label{eq00.34}
\|u\|_{W^{2}_p(\cD)}\le N\| \bar G\|_{W^{2}_p(\cD)}
+N\|g\|_{W^{2}_p(\cD)}+ N_{0} ,
\end{equation}
where $N$ depends only on $d$, $p$, $\delta$,
$R_{0}$, $K$, the  $C^{1,1}$
norm of $\partial \cD$, and $\text{\rm diam}(\cD)$
and $ N_{0} $ depends only on the same objects and
$\chi$. In particular, $ N_{0} =0$ if $\chi\equiv0$.
\end{theorem}

Here $W^{2}_p(\cD)$ denotes the set of all
functions $v$ defined in $\cD_T$ such that $v$, $Dv$, and $D^2v$ are in
$L_p(\cD)$, and $\WO^{ 2}_p(\cD_T)$ is the set of all functions $v\in
W^{2}_p(\cD)$ such that $v$ vanishes on $\partial \cD$.

In the literature, the interior $W^2_p, p>d$ estimates for a
class of fully nonlinear uniformly elliptic
equations of the form
$$
F(D^2u,x)=f(x)
$$
were first obtained by Caffarelli in \cite{Caf89} (see also \cite{CC95}). His proof is
geometric and is  based on  the Aleksandrov--Bakel'man--Pucci a
priori estimate, the Krylov--Safonov Harnack inequality and
 a covering argument which can be found
in \cite{KS80} and \cite{Sa80}.
 Adapting this
technique, similar interior estimates were proved by Wang \cite{Wa92} for
parabolic equations. In the same paper, a boundary estimate is stated but
without a proof; see Theorem 5.8 there.  By exploiting a weak reverse
H\"older's inequality, the result of \cite{Caf89} was sharpened by
Escauriaza in \cite{Es93}, who obtained the interior $W^2_p$-estimate for
the same equations allowing $p>d-\varepsilon$, with a small constant
$\varepsilon$ depending only
on the ellipticity constant and $d$.
Very recently, Winter \cite{Wi09} further extended this technique to establish
the corresponding boundary estimate as well as the
{\em $W^2_p$-solvability\/} of
the associated boundary-value problem. It is also worth noting that a
solvability theorem in the space $W^{1,2}_{p,\text{loc}}(Q)\cap C(\bar
Q)$ can be found in \cite{CKS00} for the boundary-value problem for fully
nonlinear parabolic equations. In these papers, a small oscillation
assumption in the integral sense is imposed on the operators; see, for
instance,
\cite[Theorem 1]{Caf89}. However, as pointed out in
\cite[Remark 2.3]{Wi09} and in \cite{Kr10}
(see also \cite[Example 8.3]{CKS00} for a relevant discussion),
this assumption turns out to be equivalent to a small oscillation condition in the $L_\infty$ sense,
which, particularly in the
{\em linear\/} case, is the same as what is required in the
classical $L_{p}$ theory based on the Calder\'on--Zygmund
estimates. Thus, it seems to us that the results in
\cite{Caf89,Wa92,Es93,CKS00,Wi09} mentioned above are in general not
formally applicable
to the operators under Assumption
\ref{assump1} or
\ref{assump2}, in which local oscillations are measured in the average sense so
that huge jumps in the $L_\infty$ norm are allowed.
It is still possible that the {\em methods\/} developed in the above
cited articles can be used to obtain our results. In our opinion,
our method is somewhat simpler and leads to the results faster.

The results obtained in this article  contain and
generalize the Sobolev space theory of linear equations
 with VMO coefficients, which was developed about twenty years ago  by Chiarenza,
Frasca,
and Longo in  \cite{CFL1,CFL2} for non-divergence form elliptic equations, and
later in \cite{BC93} by Bramanti and Cerutti for parabolic equations. The proofs
in these references are based on the Calder\'on--Zygmund theorem and the
Coifman--Rochberg--Weiss commutator theorem. For further related results, we
refer the reader to the book \cite{MaPaSo00} and reference therein.

However, remarkably not all known results related to VMO coefficients
and second-order elliptic and parabolic {\em linear\/} equations
can be obtained from the results of the present article.

The reader can find in \cite{Krylov_2005,Krylov_2007_mixed_VMO}  a
unified approach to investigating the $L_p$ (and $L_q-L_p$) solvability of both
divergence and non-divergence form parabolic and elliptic equations with
leading coefficients that are in VMO in the spatial
variables and only {\em
measurable\/} in the time variable in the parabolic
case. In the nonlinear setting, it is an extremely challenging problem
whether or not one can treat $F$'s which are only measurable in $t$.
The proofs in \cite{Krylov_2005,Krylov_2007_mixed_VMO} rely mainly
 on pointwise estimates of sharp
functions of spatial derivatives of solutions, so that
VMO coefficients are treated in a rather straightforward manner.
This approach is rather flexible: it has been
applied to both divergence and non-divergence form
linear equations/systems with coefficients which are very
irregular in some of the independent variables. For example, in \cite{KK1,KK2} Kim and Krylov
established the solvability in Sobolev spaces of non-divergence elliptic and parabolic equations
with leading coefficients measurable in a space variable and VMO in the other variables; in \cite{DK10} Dong and Kim considered both divergence and non-divergence form higher-order elliptic and parabolic systems in the whole space, the half space and bounded domains with coefficients in the same class as in \cite{Krylov_2005,Krylov_2007_mixed_VMO}; see also the references in \cite{DK10} for other results in this line of research.

Here we follow the general scheme in
\cite{Krylov_2005,Krylov_2007_mixed_VMO}
to study fully nonlinear elliptic and parabolic
equations in bounded domains or cylinders with VMO coefficients. This
article is a continuation of \cite{Kr10}, in which interior estimates for
elliptic Bellman's equations were obtained. The key ingredients in our
proofs are the Evans--Krylov theorem applied to homogeneous equations
with constant coefficients and a $W^2_\varepsilon$ estimate for
elliptic equations with measurable coefficients,
which is originally due to
 F.H. Lin
\cite{Li86} and extended to the parabolic case
in \cite{Kr10}. We also
remark that as in \cite{Es93,CKS00,Wi09}, by making use of a refined
Aleksandrov--Bakel'man--Pucci estimate instead of the classical estimate,
one can extend the range of $p$ in our results to $p>d-\varepsilon$ in
the elliptic case and to $p>d+1-\varepsilon$ in the parabolic case, where
$\varepsilon$ is a small constant depending only on $d$ and $\delta$.
These ranges are sharp, as is seen from
the examples in Section I.2 of \cite{LU}.

\begin{remark}
                                        \label{remark 7.30.1}
A few comments on the structures of \eqref{7.29.1}
and \eqref{7.29.2} are in order. Usually, the last two terms
on the left-hand side of  \eqref{7.29.1}
are combined into one $H=F+G$. However, if
we are given a function $H(u'',u',u,t,x)$, we can always
represent it as
$F+G$ with $F=H(u'',0,0,t,x)-H(0,0,0,t,x)$ and $G=H-F$.
Then usual ellipticity, convexity in $(u_{ij})$,
Lipschitz continuity,  and growth conditions with respect
to $(u'',u',u)$ from the
theory of fully nonlinear equations will transform into our conditions
even with $\chi\equiv0$.
Our form may look more attractive in the sense that no
convexity condition with respect to $ u''$ is
imposed on $G$. The above decomposition of $H$ lacks however
the requirement that $F$ be positive homogeneous of degree one.
Then one defines
$$
\hat{F}(u'',t,x)=\lim_{\lambda\to\infty}\frac{1}{\lambda}F
(\lambda u'',t,x),\quad \hat{G}=F-\hat{F}
$$
and combines $\hat{G}$ with $G$. The fact that $\hat{F}$ is well defined
follows from the Lipschitz continuity and convexity of $F$ in $u''$.
That $\hat{F}$ is  positive homogeneous of degree one is obvious.
Furthermore, for each $(t,x)$, the functions $\frac{1}{\lambda}F
(\lambda u'',t,x)$ are equicontinuous in $u''$, and hence
converge uniformly on compact sets which means exactly that
$$
\chi(u'',t,x):=\frac{1}{|u''|}|\hat{F}(u'',t,x)
-F ( u'' ,t,x)|\to0
$$
as $|u''|\to\infty$.
\end{remark}

\begin{remark}
                                        \label{remark 7.30.1b}
There are natural and essentially unique candidates
for the functions $\bar{F}$ in Assumptions
\ref{assump1} and \ref{assump2}. To show them
for a function $f$ defined on a Borel set $\cU \subset \bR^{d+1}$, we set
$$
(f)_{\cU} = \frac{1}{|\cU|} \int_{\cU} f(t,x) \, dx \, dt
= \dashint_{\cU} f(t,x) \, dx \, dt,
$$
where $|\cU|$ is the
$d+1$-dimensional Lebesgue measure of $\cU$.
In case $\cU$ is a Borel subset of $\bR^d$,
we define $|\cU|$ and
$(f)_{\cU}$ in a similar way.
The reader understands that if
$f$ also depends on $u''$: $f(u'',t,x)$, then after
averaging with respect to $(t,x)$ we will get
the result depending on $u''$ as well, which
we denote $(f)_{\cU}(u'')$.
Now it is easy to see that if \eqref{7.30.2} holds
with an $\bar{F}$, then it also holds with
$$
\bar{F}(u'')=(F)_{Q_{r}(\tau,z)}(u'')
$$
provided that we multiply the right-hand side of
\eqref{7.30.2} by  a constant depending only on $d$.
Thus defined $\bar{F}(u'')$ satisfies $({\rm H}_{1})$
as long as $F$ does.

\end{remark}

\begin{remark}
                                        \label{remark 7.29.2}

A typical example when it is relatively easy
to verify our hypotheses is given by the following
Bellman's equation:
$$
\partial_t u(t,x) + \sup_{\omega\in \Omega} [a^{ij}(\omega,t,x)
 D_{ij} u(t,x)+b^{i }(\omega,t,x)D_{i}u(t,x)
$$
\begin{equation}
                                                \label{7.14.2}
-c(\omega,t,x)u(x)    + f(\omega, t, x)]=0,
\end{equation}
  where
   the set $\Omega$ is
 a separable metric space, $a =(a^{ij} )$,
$b =(b^{i} )$, $c\geq0 $, and $f $
are given functions which are measurable in $(t,x)$ for
 each $\omega\in \Omega$
 and  continuous in
$\omega$ for each $(t,x)$.

 As usual, the summation convention is
enforced throughout the article and the summation in   \eqref{7.14.2}
and in similar situations is performed before the supremum is taken.
Equations of that type appear in many applications  and,
in particular, in the theory of
optimal control  of diffusion type processes they are the
so-called Bellman's equations.

Introduce,
$$
F(u'',t,x)=\sup_{\omega\in \Omega}  a^{ij}(\omega,t,x)u''_{ij},
\quad G(u'',u',u,t,x)
$$
$$
=\sup_{\omega\in \Omega} [a^{ij}(\omega,t,x)
u''_{ij}+b^{i }(\omega,t,x)u'_{i}
-c(\omega,t,x)u    + f(\omega, t, x)]-F(u'',t,x)
$$
and assume that for any $\omega$ the function $a^{ij}(\omega,t,x)u''_{ij}$
satisfies $({\rm H}_{1})$ and the function
$b^{i }(\omega,t,x)u'_{i}
-c(\omega,t,x)u    + f(\omega, t, x)$ satisfies $({\rm H}_{2})$.  Then
$F$  and $G$ satisfy (${\rm H}_1$)-(${\rm H}_3$)  with
$\chi\equiv 0$.

One can give several conditions in terms
of $a^{ij}$, which are sufficient for
\eqref{7.30.2} to hold. For instance, \eqref{7.30.2} is satisfied  if for
any
$r\in(0,R_{0}]$, $t\in\bR$, and $z\in \cD $ one can find
functions $\bar{a}^{ij}(\omega)$ such that the functions
$\bar{a}^{ij}(\omega)u''_{ij}$ satisfy $({\rm H}_{1})$ and for any
$u''\in\cS$ with $|u''|=1$
$$
\int_{Q_{r}(\tau,z)}
|\sup_{\omega}a^{ij}(\omega,t,x)u''_{ij}
-\sup_{\omega}\bar{a}^{ij}(\omega)u''_{ij}|\,dx\,dt
\leq \theta
r^{ d+2}
$$
or, since the difference of supremums is less than the supremum
of the absolute values of the differences, if for all $i,j$
\begin{equation}
                                                       \label{8.5.2}
\int_{Q_{r}(\tau,z)}
\sup_{\omega}| a^{ij}(\omega,t,x)-
 \bar{a}^{ij}(\omega) |\,dxdt
\leq \theta r^{ d+2}.
\end{equation}
In addition, if $\Omega$ is a finite set, then one can
drop the last
supremum and require the condition to hold for each $\omega$.
As in Remark \ref{remark 7.30.1b}, the latter condition
holds with some $\bar{a}$ if and only if it holds
(with slightly modified right-hand side) with
$\bar{a}=a_{Q_{r}(\tau,z)}$.
\end{remark}

The remainder of the article is organized as follows. We consider
elliptic equations in the half space with constant
coefficients in Section \ref{sec2} and with VMO coefficients in Section
\ref{sec3}. With these preparations, the proof of Theorem
\ref{mainthm2} is given in Section \ref{SecproofThm2}.
Then we turn to parabolic  equations in the whole space
with constant coefficients in Section \ref{sec4} and with VMO coefficients in
Section \ref{sec5}, as well as parabolic   equations in the half space in
Sections \ref{sec6} and
\ref{sec7}. Finally, the proof of Theorem \ref{mainthm1} is presented at the end
of Section \ref{sec7}.  The reader may notice that
we could have somewhat shortened the article by deriving some results
for elliptic equations from their parabolic counterparts.
We do not do that because it is much easier and shorter to explain
the main ideas in the elliptic case.

A few times in the article we will be using
known results from $C^{2+\alpha}$ theory of elliptic and parabolic
fully nonlinear equations. Part of these results is proved for $H$
concave in $u''$ and part for convex $H$. The reader understands
that   results for concave $H$ are also applicable
for equations with convex $H$ since the transformation $H(u'')
\to-H(-u'')$ changes the direction of convexity and does not affect
the ellipticity condition.

\mysection{Elliptic equations
with constant coefficients in $\bR^{d}_{+}$}
                            \label{sec2}

First we introduce a few more notation.
Set $\bR^{d}_{ +}=\{x\in\bR^{d}:x^{1}>0\}$.
For $r>0$ and $x=(x^1, x')\in \bR^{d}_{ +}$, denote
$$
  B_r=B_r(0), \quad B_r(x^1)=B_r(x^1,0),
$$
$$
B_r^{+}(x)=B_r(x)\cap \bR^{d}_{+}, \quad B_r^{+}= B_r^{+}(0),
\quad B_r^{+}(x^1)=B_r^{+}(x^1,0).
$$
Recall that by $Du$ and $D^{2}u$
we denote the gradient and the
Hessian of
$u$, respectively.

In this section, we are interested in the equation
\begin{equation}                                     \label{3.6.1}
F(D^{2}u )=f(x),
\end{equation}
in the half space $\bR^d_+$ with $F=F(u'',x)$ independent of $x$.
Since $F$ is convex and positive homogeneous of degree one,
it has a representation as in Remark \ref{remark 7.29.2},
so that we are dealing with Bellman's equations.

\begin{lemma}
                     \label{3.6.2}
For any $u\in  W^2_d(B_r^+) $  vanishing on $x^1=0$,
we have
$$
\sup_{B_r^{+}} |u(x)- x^{1}(D_{1}u )_{B_r^{+}}|^d\leq N
 r^{2d}\dashint_{B_r^{+}} |D^2u|^d \,dx,
$$
where $N$ depends only on $d$.
\end{lemma}

\begin{proof}
Let  $\tilde u$ be the odd extension of $u$ with respect to $x^1$, i.e.,
 $\tilde {u}(x^1,x'):= u(|x^1|,x')  \text{sgn}(x^1)$. By Lemma 8.2.1
in \cite{Kr08}, $\tilde{u}\in W^{2}_{d}$.  Note that
$$
( \tilde u )_{B_r}=0, \quad ( D_{1}\tilde{u} )_{B_r}=
(D_{1}u )_{B_r^{+}},\quad ( D_{i}\tilde{u}  )_{B_r}=0 \quad \text{for} \quad i\geq 2.
$$
 The lemma then  follows from Lemma 2.1 of \cite{Kr10}.
\end{proof}

\begin{lemma}
                 \label{3.6.3}
Let $r\in(0,\infty)$, $\kappa\geq2$ and let
$v\in C (\bar{B}_{\kappa r}^+)
\cap C^{2}_{b}(B^{+}_{\kappa\rho})$ for any $\rho\in(0,r)$. Assume that
$v$ is a solution of  \eqref{3.6.1}  in $  B_{\kappa r}^+ $ with
$f\equiv 0$ and $v=0$ on $x^1=0$. Then there are
constants
$\alpha \in (0,1)$ and N, depending only on $d$ and $\delta$, such that
$$
[ D^2v]_{C^\alpha(B^+_r)}\le N
(\kappa r)^{-2-\alpha} \sup_{\partial B_{\kappa r}^{+}}|v|.
$$

\end{lemma}
\begin{proof}
Dilations show that it suffices to prove the inequality for $\kappa r=1$. In this case,
the result follows from Theorems 7.1 of \cite{SA}
or of \cite{Sa94}, which state  that
$$
 [D^2 v ]_{ C^\alpha(B^+_{1/2})} \le N  \sup_{B^{+}_1} |v|.
$$
Due to the maximum principle, the lemma is proved.
\end{proof}

Denote by
$\cS_{\delta}$  the set of symmetric
$d\times d$-matrices $\alpha=(\alpha^{ij})$ satisfying
$$
 \delta|\xi|^2 \leq \alpha^{ij}\xi_i\xi_j \leq \delta^{-1}|\xi|^2,
\quad \forall \xi\in\bR^{d} .
$$
Introduce $\bL_\delta$ as the collection of operators
$Lu=a^{ij}D_{ij}u$ with $a(x)
=(a^{ij}(x))\in \mathcal{S}_{\delta}$
for all $x\in\bR^{d}$.

 We need a slight generalization
of  the main result of \cite{Li86} (stated as Lemma 2.3 in \cite{Kr10})
which can be proved in the same way as in \cite{Li86}
by using dilations and  standard approximation  arguments.
\begin{lemma}
                                                  \label{lemma 8.11.1}

Let $r\in(0,\infty)$ and let
  $u\in   C(\bar B_{r})\cap  W^{ 2}_{d }(B_{\rho})$
for any $\rho\in (0,r)$. Then there are constants $\gamma\in(0, 1]$ and $N$,
depending only on
$\delta, d$ such that
for any $L\in  \bL_\delta $ we have
$$
\dashint_{B_{r}} |D^2u|^\gamma \, dx   \leq N
\left(\dashint_{B_{r}}| Lu|^{ d } \,
dx \right)^{\gamma/ d }+Nr^{-2\gamma}\sup_{\partial
B_{r}} |u|^\gamma .
$$
\end{lemma}

\begin{lemma}
                        \label{3.6.5}
Let $r\in(0, \infty)$ and let  $w \in
W^{2}_{d}(B^{+}_{\rho})\cap C (\bar B_{r}^{+})$ for any $\rho\in(0,r)$.
Assume that $w=0$ on $\partial B_{r}^{+}$. Then
there are constants $\gamma\in(0,1]$ and $N$, depending only
on $\delta$ and $d$, such that for any $L\in\bL_\delta$,
$$
\dashint_{B^{+}_{r}} |D^2 w|^\gamma \,dx \leq N
 \left(\dashint_{B_r^{+}}|Lw|^d\,dx\right)^{\gamma/d}.
$$
\end{lemma}

\begin{proof}
Denote $f=Lw$. Let $\tilde w$ and $\tilde f$ be the
odd extension of $w$ and $f$ with respect to $x^1$.
Denote by $\tilde L\in \bL_\delta$
the operator with coefficients
$$
\tilde a^{ij}(x)={\text{sgn}(x^1)}a^{ij}(|x^1|,x')\quad \text{for}\,\,i=1,j\ge 2\,\,\text{or}\,\,j=1,i\ge 2,
$$
$$
\tilde a^{ij}(x)=a^{ij}(|x^1|,x')\quad \text{otherwise}.
$$
Clearly, $\tilde w\in  C(\bar B_{r})\cap  W^{ 2}_{d }(B_{\rho})$
for any $\rho<r$, $\tilde{w}=0$ on $\partial B_{r}$, and
$\tilde L\tilde w=\tilde f$
in $B_r$. Now Lemma \ref{lemma 8.11.1} yields
$$
\dashint_{B_{r}} |D^2 \tilde w|^\gamma \,dx \leq N
\left(\dashint_{B_r}|\tilde f|^d\,dx\right)^{\gamma/d}.
$$
To finish the proof of the lemma, it
suffices to recall the definitions of $\tilde w$ and $\tilde f$.
\end{proof}

Everywhere below in this section
$\alpha$ is the constant from Lemma \ref{3.6.3}
and $\gamma$ is the one from Lemma \ref{3.6.5}.

\begin{lemma}
                                                   \label{osc}
Let $r\in(0,\infty)$, $\kappa\geq  16 $, $x_0^1\geq0$.
 Let $u\in  W^2_d
(  B^+_{\kappa r} (x_0^1)) $ be a solution of  \eqref{3.6.1}
in $   B^+_{\kappa r} (x_0^1) $  vanishing on
 $B_{\kappa r}(x_0^1)\cap \partial \bR^d_+$. Then
$$
\dashint_{B^{+}_r(x_0^1)}\dashint_{B^{+}_r(x_0^1)}| D^2 u (x)-
 D^2 u (y)|^\gamma \,dx\,dy
$$
\begin{equation}
                                                  \label{osc1}
\leq N\kappa^d \left(\dashint_{B^{+}_{\kappa r}
(x_0^1)} |f|^d \,dx\right)^{\gamma/d}
 +N\kappa^{-\gamma \alpha} \left(\dashint_{B^{+}_{\kappa r}(x_0^1)}
 |  D^2 u|^d \,dx\right)^{\gamma/d},
\end{equation}
where  the constant  $N$ depends only on $d$ and $\delta$.

\end{lemma}
\begin{proof}
 Dilations show that it suffices to prove
the lemma only for $\kappa r=8$. We consider two cases.

{\em Case 1: $x_0^1>1$.} In this case, we have $B^+_{\kappa r/8}(x_0^1)
=B_{ r\kappa/8}(x_0^1)\subset \bR^d_+$. Therefore, inequality
\eqref{osc1} is an immediate consequence of Lemma 2.4 in
\cite{Kr10} since
$\kappa/8\ge 2$  (cf. the comment  at  the beginning of the section).

{\em Case 2: $x_0^1\in [0,1]$.} Since $r=8/\kappa\leq1/2$, we have
$$
B^{+}_r(x_0^1)\subset B^{+}_2 \subset B^{+}_4\subset B^{+}_{\kappa r}(x_0^1).
$$
By using a standard density argument, we may assume
$u\in C_b^{\infty}(\bar{B}_{\kappa r}^{+}(x_0^1))$.
 Define $\hat{u}(x):= u(x)- x^{1}(D_{1}u )_{ B _{4}^{+}}$.
Let $v$ be a classical solution of \eqref{3.6.1} in $B^+_4$
 with $f\equiv 0$
and boundary condition $v=\hat u$ on $\partial B _{4}^{+} $.
Such a solution exists due to Theorems 7.1 of \cite{SA}
 or of \cite{Sa94}.
Then by Lemmas \ref{3.6.3} and~\ref{3.6.1},
\begin{eqnarray*}
\dashint_{B^{+}_r(x_0^1)}\dashint_{B^{+}_r(x_0^1)}|  D^2v (x)-
  D^2 v (y)|\,dx\,dy
\le Nr^\alpha[  D^2v]_{ C^{\alpha}(B^{+}_2)}
                 \nonumber\\
\le Nr^\alpha\sup_{\partial B^+_4} |v|
 =Nr^\alpha\sup_{\partial B^+_4} |\hat u|
\leq N \kappa^{-\alpha}
\left(\dashint_{B^{+}_4} |D^2 u|^d \,dx\right)^{1/d}.
\end{eqnarray*}
Recall that $\gamma\in (0,1]$. By H\"older's inequality, we get
\begin{eqnarray}
                                                  \label{eqn3.6.6}
\dashint_{B^{+}_r(x_0^1)}\dashint_{B^{+}_r(x_0^1)}|
   D^2v (x)-   D^2 v (y)|^\gamma\,dx\,dy
                 \nonumber\\
\le N \kappa^{-\gamma\alpha}
\left(\dashint_{B^{+}_{\kappa r}(x_0^1)} | D^2 u|^d \,dx\right)^{\gamma/d}.
\end{eqnarray}

Next we recall a simple and well-known fact that condition $({\rm H}_{1})$
implies that for any
$\cS$ valued functions $ u''(x) $ and
$ v''(x) $ there is an operator $L=a^{ij}D_{ij}\in \bL_{\delta}$ such that
$F(u''(x))-F(v''(x))=a^{ij}[u''_{ij}-v''_{ij}](x)$.
Then set
 $w:= \hat{u}-v$ in ${B^{+}_4}$ and notice that
$w\in W^{2}_{d}(B^{+}_{\rho})\cap C (\bar B_{4}^{+})$
for any $\rho<4$,  $w=0$ on
$\partial B_4^+$, and $F(D^{2}\hat{u})=f$.

It follows by the above that there exists an operator $L\in \bL_\delta$  such
that
$Lw =f$ in $B^{+}_4$. By Lemma \ref{3.6.5}
and the fact that $\kappa r=8$, we get
$$
\dashint_{B^{+}_r(x^1_0)}|  D^2w|^{\gamma}\,dx
\leq N \kappa^{d}\dashint_{B^{+}_4}|  D^2w|^\gamma\,dx
\leq N \kappa^d\left(\dashint_{B^{+}_{\kappa r}(x^1_0)}
 |f|^d\,dx\right)^{\gamma/d}
$$
and
\begin{eqnarray*}
\dashint_{B^{+}_{r}(x^1_0)}\dashint_{B^{+}_{r}(x^1_0)}|  D^2w (x)-  D^2w (y)|^{\gamma}\,dx\,dy\\
\leq N \kappa^d\left(\dashint_{B^{+}_{\kappa r}(x^1_0)}
|f|^d\,dx\right)^{\gamma/d}.
\end{eqnarray*}
Combining this with \eqref{eqn3.6.6}
and observing that $D^{2}u=D^{2}v+D^{2}w$ yield \eqref{osc1}. The lemma
is proved.
\end{proof}

If $g$ is a measurable function in $\bR^{d} $, define
 its maximal function by
$$
\mathbb{M}(g)(x)=\sup_{ B_{r}(y)\ni x }
\dashint_{ B_{r}(y) } |g(z)| \,dz.
$$
It is easy to see that, for any $r>0$ and $x \in\bR^{d}_{+}$,
  we have
\begin{equation}
                                                    \label{7.13.1}
\dashint_{B^{+}_{r}(x )}|g(z)|\,dz\leq 2
\dashint_{B _{r}(x )}|g(z)I_{\bR^{d}_{+}}(z)|\,dz
\leq 2\bM(gI_{\bR^{d}_{+}})(x).
\end{equation}

Next in the measure
space $\bR^{d}_{+}$ endowed with the Borel $\sigma$-field and Lebesgue
 measure
consider the filtration of dyadic cubes $\frC=\{\bC_{n},n\in\bZ\}$, where
$\bZ=\{0,\pm1,\pm2,...\}$ and $\bC_{n}$ is the collection of cubes
$$
 (i_{1}2^{-n},(i_{1}+1)2^{-n}]\times...\times
(i_{d}2^{-n},(i_{d}+1)2^{-n}],\quad i_{1},...,i_{d}\in\bZ,\,\,i_1\ge 0.
$$
 For $x\in\bR^{d}_{+}$ introduce
$$
g^\#_{\gamma}(x)=
\sup_{C\in\frC\,:\,  x\in C}
\left(\dashint_{C} \dashint_{C}|g(y)-g(z)|^\gamma
\,dydz\right)^{1/\gamma}.
$$
Notice that if $x\in C\in\frC$, then for the smallest $r>0$
such that $C\subset B_{r}(x)$ we have
$$
\dashint_{C} \dashint_{C}|g(y)-g(z)|^\gamma
\,dydz\leq N(d)
\dashint_{B^{+}_{r}(x)} \dashint_{B^{+}_{r}(x)}|g(y)-g(z)|^\gamma
\,dydz.
$$
This along with \eqref{7.13.1} and Lemma \ref{osc} lead to the following.
\begin{corollary}
                                                \label{corollary 7.13.1}

 Let $u\in  \WO^2_d
(\bR^{d}_{+}) $ be a solution of  \eqref{3.6.1}
in $\bR^{d}_{+}$. Then, for any $x\in\bR^{d}_{+}$ and $\kappa\geq16$,
 we have
$$
(D^{2}u)^\#_{\gamma}(x)\leq N\kappa^{d/\gamma}
\bM^{1/d}(|f|^{d}I_{\bR^{d}_{+}})(x)+N\kappa^{-\alpha}
\bM^{1/d}(|D^{2}u|^{d}I_{\bR^{d}_{+}})(x),
$$
where the constant $N$ depends only on $d$ and $\delta$.
\end{corollary}

Now we recall Theorem 5.3 of
\cite{Kr10} which is a version of the Fefferman--Stein theorem:
Let
$p\in(1,\infty)$  and $\gamma\in (0,1]$. Then for any
$g\in L_{p}( \bR^{d}_{+} )$, we have
\begin{equation}                                         \label{fs}
\|g\|_{L_{p}( \bR^{d}_{+} )}
\leq N\| g^{\#}_\gamma \|_{L_{p}( \bR^{d}_{+} )},
\end{equation}
where $N$ depends on $p, \gamma $, and $d$ only.

\begin{theorem}
                         \label{3.22.1}
 Let $p>d$.  (i) Let $u\in \WO^{2}_p(\bR^{d}_{+})$ satisfy
\eqref{3.6.1}. Then there exists $N=N(d,\delta, p)$ such that
\begin{equation*}
\|D^2u\|_{L_p(\bR^{d}_{+})} \leq N \|f \|_{L_p(\bR^{d}_{+})}.
\end{equation*}

(ii)   For any $\lambda > 0$ and
$u\in \WO^{2}_p(\bR^{d}_{+})$, we have
\begin{equation}
                       \label{3.6.7}
\lambda\|u\|_{L_p(\bR^{d}_{+})} +\|D^2u\|_{L_p(\bR^{d}_{+})} \leq N \|
F(D^{2}u)-\lambda u
\|_{L_p(\bR^{d}_{+})},
\end{equation}
where $N$ depends only on $d, p$, and $\delta$.

(iii) For any
$f \in L_p(\bR^{d}_{+})$ and $\lambda>0$,
there is a unique solution
$u\in \WO^{2}_p(\bR^{d}_{+})$ of the equation
\begin{equation}
                            \label{4.5.1}
F(D^{2}u)-\lambda u=f.
\end{equation}
\end{theorem}

\begin{proof}
(i) First fix $\kappa \geq  16 $.
It follows from Corollary \ref{corollary 7.13.1},
\eqref{fs}, and the Hardy--Littlewood theorem on maximal functions that
\begin{equation}
                                                          \label{7.13.3}
\|D^2u\|_{L_p(\bR^{d}_{+})}\leq N \kappa^{d/\gamma}
 \|f\|_{L_p(\bR^{d}_{+})} + N \kappa^{-\alpha}
\|D^2u\|_{L_p(\bR^{d}_{+})},
\end{equation}
where $N=N(d,\delta,p)$. Assertion (i) is proved once noting that the
inequality holds for arbitrary $\kappa\geq  16 $.

(ii) Assertion (i) implies that, to prove \eqref{3.6.7}, it suffices to prove
\begin{equation}
                                                              \label{7.31.3}
\lambda\|u\|_{L_p(\bR^{d}_{+})} \leq N \|f \|_{L_p(\bR^{d}_{+})},
\end{equation}
where $f=F(D^{2}u)-\lambda u$.

We may assume that $u $
is smooth in $\bar\bR^{d}_{+}$ and vanishes
for $x$ large and for  $x^1=0$. Take an operator $L\in
\bL_\delta$   such that  and $Lu -\lambda u=f$. Then we obtain \eqref{7.31.3}
by Theorem 3.5.15 and the proof of
Lemma 3.5.5 in
\cite{Kr87} with $N$ depending only on $d, p$, and  $\delta$.

(iii) The proof of the solvability of \eqref{4.5.1}
relies on its solvability in $C^{2+\alpha}(\bar{\bR}^{d}_{+})$
with zero boundary condition ($\alpha\in(0,1)$ is perhaps
different from the one above). First we assume
that
$f\in C^{\infty}_{0}(\bR^{d}_{+})$ and by using classical results
(see, for instance, \cite{Kr87} or  \cite{Tr83}) find a function
$u\in C^{2+\alpha}(\bar{\bR}^{d}_{+})$ with $u(0,\cdot)=0$
satisfying \eqref{4.5.1}. Simple barriers show that
$u(x)\to0$ exponentially fast as $|x|\to\infty,x^{1}\geq0$.

Furthermore, there is a well-known and standard procedure
(see, for instance the proof
of Lemma 2.4.4 in \cite{Kr08})
to derive from \eqref{3.6.7} that
\begin{equation}
                                                          \label{7.5.6}
\|u\|_{W^{2}_{p}(B_{1}^{+}(x))}\leq N\|u\|_{L_{p}(B_{2}^{+}(x))}
+N\|f\|_{ L_{p} (B_{2}^{+}(x))},\quad x\in\bR^{d}_{+},
\end{equation}
where $N$ is independent of $x$. To start the procedure
it suffices to notice that for
nonnegative $\zeta\in C^{\infty}_{0}(\bR^{d})$
we have
that
$$
F(D^{2}(\zeta u))-\lambda \zeta u=\zeta f+g,
$$
where
$$
g=F(D^{2}(\zeta u)) -\zeta
F(D^{2}u) ,
$$
and by the homogeneity and Lipschitz continuity of $F$
$$
|g|\leq N|D^{2}(\zeta u)-\zeta D^{2}u|\leq  N(|Du|+|u|).
$$
Upon combining \eqref{7.5.6} and the fact that $u,f\in L_{p}(\bR^{d}_{+})$,
we conclude that $u\in\WO^{2}_{p}(\bR^{d}_{+})$, so that estimate
\eqref{3.6.7} is applicable.

Having done this step we approximate the given $f\in L_{p}(\bR^{d}_{+})$
in the $L_{p}(\bR^{d}_{+})$ norm by functions $f_{n}
\in C^{\infty}_{0}(\bR^{d}_{+})$, which would give us
a sequence of $u_{n}\in\WO^{2}_{p}(\bR^{d}_{+})$
with the $\WO^{2}_{p}(\bR^{d}_{+})$ norms bounded
such that $F(D^{2}u_{n})=f_{n}$.
A subsequence $u_{n'}$ converges then uniformly on compact
subsets of $\bar{\bR}^{d}_{+}$ to a function
$u\in\WO^{2}_{p}(\bR^{d}_{+})$. That $u$ satisfies
\eqref{4.5.1} now follows from Theorems 3.5.15 and 3.5.6 (a)
of \cite{Kr87}.
This proves the existence of solutions.

As usual, uniqueness follows from the fact that
$F(D^{2}u)-F(D^{2}v)= L(u-v)$, where $L\in\bL_{\delta}$.
The theorem is proved.
\end{proof}

\mysection{Elliptic equations
with VMO coefficients in $\bR^{d}_{+}$}
                                    \label{sec3}

We are about to deal with the equation
\begin{equation}                                            \label{goal}
F(D^{2}u,x)
  -\lambda u = f(x)
\end{equation}
 in the half space $\bR^d_+$.
Of course, Assumption \ref{assump2} is supposed to hold  with
$\cD=\bR^d_+$.

\begin{remark}
                                        \label{remark 7.29.1}
We are going to use the following fact:
For any $\mu >0$ there exists $\theta=\theta(\mu, d, \delta )>0$
  such that, if
\eqref{7.28.1} holds with this $\theta$ for any $u''\in\cS$ with $|u''|=1$,
then
\begin{equation}
                                                     \label{7.30.1}
\int_{B_{r}( z)}  \sup_{u''\in\cS:|u''|=1}
|F(u'' , x)-\bar{F}(u'')|\,dx\leq\mu r^{d}.
\end{equation}

To prove this observe that one can find $n=n(d, \delta,\mu )$ points
$u''_{1},...,u''_{n}$ such that, for any  $x$ and any
$u''$ with $ |u''|=1 $, at least one of
$|F(u'',x)-F(u''_{i} ,x )|+|\bar{F}(u'')-\bar{F}(u''_{i})|$
is less than $\mu /(4|B_{1}|)$. The latter is possible
due to the Lipschitz continuity of $F$ and $\bar{F}$
in $u''$ (uniform with respect to $x$).
After that it obviously suffices to choose $\theta=\mu /(4n)$.

We are also going to use the fact that the supremum
in \eqref{7.30.1} is bounded by a constant depending only on $\delta$
and $d$.
\end{remark}

Everywhere below in this section
$\alpha$ is the constant from Lemma \ref{3.6.3}
and $\gamma$ is the one from Lemma \ref{3.6.5}.

\begin{lemma}
                                                     \label{osca}
Let $\beta\in(1,\infty)$,  $\lambda=0$,
$\kappa\geq  16 $,  $ \mu, r>0$, $x_{0}^{1}\geq0$, and $z\in
\bR^d_{+}$.  Suppose that $\theta=\theta(\mu,d  ,\delta )$.
 Let $u\in \WO^{2}_{d}(\bR^{d}_+)$ be
a solution of  \eqref{goal}  vanishing outside $B^{+}_{R_0}(z)$. Then
$$
\dashint_{B^{+}_r(x_{0}^{1})}\dashint_{B^{+}_r(x_{0}^{1})}| D^2u (x)-
 D^2u (y)|^\gamma\,dx\,dy
$$
$$
\leq N \kappa^{d}\left(\dashint_{B^{+}_{\kappa r}(x_{0}^{1})}
  |f|^d \,dx\right)^{\gamma/d}
+N\kappa^{d}\left(\dashint_{B^{+}_{\kappa r}(x_{0}^{1})}|
D^2u|^{\beta d}\,dx\right)^{\gamma/(\beta{d})}\mu^{
\gamma/(\beta'd)}
$$
$$
+N\kappa^{-\gamma\alpha}\left(
\dashint_{B^{+}_{\kappa r}(x_{0}^{1})}| D^2u|^d\,dx\right)^{\gamma/d},
$$
where $N=N(d,\delta ,\beta )$ and $\beta'=\beta/(\beta-1)$.
\end{lemma}
\begin{proof}
 Introduce
 $$
\bar{F}(u'')=\left\{
               \begin{array}{ll}
                 (F)_{B^{+}_{R_0}(z)}(u''), & \hbox{if $\kappa r\geq R_0$;} \\
                 (F)_{B^{+}_{\kappa r}(x_{0}^{1})}(u''), & \hbox{otherwise.}
               \end{array}
             \right.
$$
Observe that
$$
\bar F(D^2 u)=\hat f(x),
$$
where
$$
\hat f(x)= \bar F(D^2 u)-F(D^2 u,x) +f(x).
$$
 By Lemma \ref{osc},
\begin{equation*}
\dashint_{ B^{+}_r(x_{0}^{1})} \dashint_{ B^{+}_r(x_{0}^{1})}
| D^2u (x)- D^2u (y)|^\gamma\,dx\,dy
\end{equation*}
\begin{equation*}
\leq N \kappa^{d}\left(\dashint_{ B^{+}_{\kappa r}(x_{0}^{1})}
 |\hat{ f}|^d\,dx\right)^{\gamma/d}
+N\kappa^{-\gamma\alpha}\left(\dashint_{ B^{+}_{\kappa r}(x_{0}^{1})}
| D^2u|^d\,dx\right)^{\gamma/d}.
\end{equation*}

Notice that  by the triangle inequality,
$$
\dashint_{ B^{+}_{\kappa r}(x_{0}^{1})}
 |\hat{f}|^d\,dx
\leq N\dashint_{ B^{+}_{\kappa r}(x_{0}^{1})} |f|^d\,dx
+ N\dashint_{ B^{+}_{\kappa r}(x_{0}^{1})}|\bar F(D^2 u)-F(D^2 u,x)|^d\,dx.
$$
For any $x\in \bR^d_+$, denote
$$
h(x)=\sup_{u''\in \cS: |u''|=1}|F(u'',x)-\bar F(u'')|.
$$
By H\"older's inequality,
$$
\dashint_{ B^{+}_{\kappa r}(x_{0}^{1})}|\bar F(D^2 u)-F(D^2 u,x)|^d\,dx
\le \dashint_{ B^{+}_{\kappa r}(x_{0}^{1})}h^d(x)  |D^2 u|^d\,dx\leq
J_1^{1/\beta}J_2^{ 1/\beta'},
$$
where
$$
J_1= \dashint_{ B^{+}_{\kappa r}(x_{0}^{1})} |D^2u|^{\beta d}\,dx,
$$
$$
J_2= \dashint_{ B^{+}_{\kappa r}(x_{0}^{1})} h^{\beta'
d}(x) I_{B^{+}_{R_0}(z)}\,dx
\leq N \dashint_{ B^{+}_{\kappa r}(x_{0}^{1})}  h(x)I_{B^{+}_{R_0}(z)}\,dx.
$$
If $\kappa r\geq R_0$, then by Remark \ref{remark 7.29.1}
$$
J_2\leq N(\kappa r)^{-d}\int_{B^{+}_{R_0}(z)} h(x) \,dx
\leq N(\kappa r)^{-d} R_0^d\dashint_{B^{+}_{R_0}(z)}  h(x)\,dx \leq
N\mu.
$$
If $\kappa r<R_0$, then
$$
J_2\leq N\dashint_{B^{+}_{\kappa r}(x_{0}^{1})} h(x)\,dx \leq N
\mu.
$$
Therefore,
$$
\dashint_{B^{+}_{\kappa r}(x_{0}^{1})}|\bar F(D^2 u)-F(D^2 u,x)|^d\,dx\leq
 N\left(\dashint_{B^{+}_{\kappa r}(x_{0}^{1})}
| D^2u(x)|^{\beta d}\,dx\right)^{1/\beta}\mu^{ 1/\beta'}
$$
and this leads to the desired result.
The lemma is proved.
\end{proof}

\begin{corollary}
 Under the assumptions of Lemma \ref{osca}, let $p>\beta d$.
Then there is a constant $N_{0}$, depending only on $\delta, \beta, d$, and
 $p$, such that
$$
\|D^2 u\|_{L_p(\bR^d_+)}  \leq N _{0}
\kappa^{ d/\gamma} \|f\|_{L_p(\bR^d_+)}
+N_{0}(\kappa^{ d/\gamma}\mu^{ 1/(\beta' d)} +
\kappa^{-\alpha})\|D^2 u\|_{L_p(\bR^d_+)}.
$$
\end{corollary}

Indeed it suffices to proceed as in the derivation of \eqref{7.13.3}.

By taking $2\beta=1+p/d$ and choosing $\kappa$ and $\theta$ in such a way that
$$
N_{0}(\kappa^{ d/\gamma}\mu^{1/(\beta' d)} +
\kappa^{-\alpha})\leq \frac{1}{2},
$$
we arrive at the following corollary.
\begin{corollary}
                     \label{corollary 3.6.8}
Let $p>d$,  $u\in \WO^{2}_{d}(\bR^{d}_+)$ be a solution of  \eqref{goal}
with $\lambda=0$
 vanishing outside $B^{+}_{R_0}(z)$, where $z\in \bR^d_{+}$.
Then there exist $\theta=\theta(d,p, \delta)\in (0,1]$ and
$N=N(d,p, \delta)$ such that if
 Assumption \ref{assump2} is satisfied with this $\theta$,
then $\| D^2 u\|_{L_p(\bR^{d}_+)}\leq N\|f\|_{L_p(\bR^{d}_+)}$.
\end{corollary}

 The next theorem is the main result of this section.

\begin{theorem}
                            \label{thm3.4}
Let $p>d$. Then there exist constants $\theta\in (0,1]$
depending  only on $d,p, \delta$   and $\lambda_0$ depending
only on $d,p, \delta$, and $R_0$
such that if Assumption \ref{assump2} holds with this $\theta$, then

(i) For any
$\lambda\geq  \lambda_0 $ and any $u\in \WO^{2}_p(\bR^{d}_{+})$
 satisfying  \eqref{goal}, we have
\begin{equation}
                                               \label{7.13.7}
\lambda\|u\|_{L_p(\bR^{d}_{+})}+\|D^2u\|_{L_p(\bR^{d}_{+})}\leq
 N\| f \|_{L_p(\bR^{d}_{+})},
\end{equation}
where $N=N(d, p, \delta)$;

(ii)
For any $\lambda>0$, there exists a constant
 $N=N(d, p, \delta, R_0,\lambda)$
such that  if $u,v\in W^{2}_p(\bR^{d}_{+})$ and  $u-v\in
\WO^{2}_p(\bR^{d}_{+})$, then
\begin{equation}
                                \label{eq12.53}
\|u\|_{W^2_p(\bR^{d}_{+})} \leq N \|F(D^{2}u,
\cdot )-\lambda u
\|_{L_p(\bR^{d}_{+})}+ N\|v\|_{W^2_p(\bR^{d}_{+})},
\end{equation}
\end{theorem}

\begin{proof} We suppose that
Assumption \ref{assump2} holds with
$\theta$ from Corollary~\ref{corollary 3.6.8}.

(i) Take
a nonnegative $\zeta\in C^{\infty}_{0}$ which
 has support in $B^{+}_{R _{0}}$
and is such that $\zeta^{p}$ integrates to one.
 For the parameter $z\in\bR^{d}_{+}$
define
$$
u_{ z}(x)=u(x)\zeta(z-x)
$$
and observe that for any   $x\in\bR^{d}_{+}$ we have
\begin{equation}
                                             \label{7.13.4}
\int_{\bR^{d}_{+}}\zeta^{p}(z-x)\,dz=1.
\end{equation}
Then notice that,  by the homogeneity of $F$,
for any $z\in\bR^{d}_{+}$
$$
F(D^2 u_z(x),x) =
f_{z}(x) +\lambda u_z(x),
$$
where
$$
f_{z}(x)=f (x)\zeta(z-x)+
F(D^2 u_z(x),x)-F(\zeta(z-x)D^2 u,x)
$$

By Corollary \ref{corollary 3.6.8} and the
Lipschitz continuity of $F$ in $u''$,
$$
\|\zeta(z-\cdot)|D^{2}u|\|_{L_p(\bR^{d}_{+})}^{p}
\leq N\|\zeta(z-\cdot)f\|_{L_p(\bR^{d}_{+})}^{p}
$$
$$
+N\||D\zeta(z-\cdot)|\,|Du|\,\|_{L_p(\bR^{d}_{+})}^{p}
+N\|(|D^{2}\zeta(z-\cdot)|+\lambda \zeta(z-\cdot))u\|_{L_p(\bR^{d}_{+})}^{p}.
$$
Upon integrating through this estimate
with respect to $z\in \bR^d_+$
and using \eqref{7.13.4} we get
$$
\| D^{2}u \|_{L_p(\bR^{d}_{+})}^{p}
\leq N_{1}(\| f\|_{L_p(\bR^{d}_{+})}^{p}+\lambda^{p}
\| u\|_{L_p(\bR^{d}_{+})}^{p} )
$$
$$
+N_{2}(\|   Du \|_{L_p(\bR^{d}_{+})}^{p}
+\| u\|_{L_p(\bR^{d}_{+})}^{p}),
$$
where $N_{1}=N_{1}(d,\delta,p)$ and $N_{2}=N_{2}
(d,\delta,p, R_{0})$. Furthermore, as in the proof of Theorem
\ref{3.22.1}, by analyzing the proof of Lemma 3.5.5 of \cite{Kr87},
 we have for any $\lambda>0$
$$
 \lambda \|u\|_{L_p(\bR^{d}_{+})} \leq
N\|f\|_{L_p(\bR^{d}_{+})},
$$
where $N$ depends only on $d, p$, and  $\delta$. Hence
$$
\lambda^{p} \|u\|^{p}_{L_p(\bR^{d}_{+})}+\| D^{2}u
\|_{L_p(\bR^{d}_{+})}^{p}
\leq N_{1}\| f\|_{L_p(\bR^{d}_{+})}^{p}
+N_{2}(\|   Du \|_{L_p(\bR^{d}_{+})}^{p}
+\| u\|_{L_p(\bR^{d}_{+})}^{p}),
$$
and to finish proving \eqref{7.13.7} with $N=2N_{1}$ it only remains to
use the multiplicative inequalities and choose
$\lambda_{0}(d,\delta,p,R_{0})$ sufficiently large.

(ii) Set $w=u-v$ and $f=F(D^{2}u ,x )-\lambda u$. Observe that
$$
F(D^{2}w ,x )-\lambda w=f +[F(D^{2}w ,x )-
F(D^{2}w +D^{2}v ,x  )]+\lambda v
$$
and $| F(D^{2}w ,x )-F(D^{2}w +D^{2}v ,x  )|\leq N|D^{2}v|$.
Then we see that \eqref{eq12.53} follows from
the proof of Assertion (i).
The theorem is proved.
\end{proof}

 The following solvability theorem
is a standard result, which however will not be used
later in the paper. The main emphasis here is on the method
of proof.
\begin{theorem}
                            \label{theorem 8.5.1}
Let $p>d$. Then there exist constants $\theta\in (0,1]$
depending  only on $d,p, \delta$   and $\lambda_0$ depending
only on $d,p, \delta$, and $R_0$
such that if Assumption \ref{assump2} holds with this $\theta$, then
for any $v\in W^{2}_p(\bR^{d}_{+})$, $\lambda>0$,
and $f \in L_p(\bR^{d}_{+})$,
there exists a
unique
$u\in W^{2}_p(\bR^{d}_{+})$ satisfying   \eqref{goal}   and such that
$u-v\in
\WO^2_p(\bR^{d}_{+})$.

\end{theorem}

\begin{proof}

We take $\theta$  from Theorem \ref{thm3.4}
and first we assume that $F(u'',x)$ is
infinitely  differentiable with respect to $x\in\bR^{d}$ and each of its
derivatives is
continuous in $(u'',x)$ and Lipschitz continuous
in $u''$ uniformly with respect to $x$
(in particular, if in addition $F_{x^{k}}$ are
differentiable with respect to $u''$ for $u''\ne0$, then
$F_{x^{k}u_{ij}}$ are bounded for $u''\ne0$).
  By mollifying $F(u'',x)$ with respect to $u''$ and
using its  properties listed above we obtain a sequence $F^{n}(u'',x)$
of functions infinitely differentiable in $(u'',x)$,
converging to $F$ as $n\to\infty$,
convex in $u''$ and such that, for all $n$ and all values of the arguments
and $\xi\in\bR^{d}$ and $v''\in\cS$, we have
$$
\delta|\xi|^{2}\leq F^{n}_{u_{ij}}\xi^{i}\xi^{j}
\leq\delta^{-1}|\xi|^{2},\quad |F^{n}-F^{n}_{u_{ij}}u_{ij}|\leq 1,
$$
$$
-F^{n}_{u_{ij}x^{k}}v''_{ij}\xi^{k}-F^{n}_{ x^{k}x^{r}}\xi^{k}\xi^{r}
\leq N(|v''| +|\xi| )|\xi|,
$$
where $N$ is independent of $n$.
It follows that the function
$-F(-u'',x)-\lambda u$ is of class $\bar{\cF}$ introduced in \cite[\S 6.1]{Kr87}.
We also take
$f\in C^{\infty}_{0}(\bR^{d}_{+})$ and $v\in C^{\infty}(\bR^{d})$ with
compact support. Then by classical results
(see, for instance, \cite{Kr87} or  \cite{Tr83}) equation  \eqref{goal}
with boundary condition $u=v$ on $\partial\bR^{d}_{+}$ has
a unique solution $u$, which is twice continuously differentiable
in $\bar{\bR}^{d}_{+}$. As in the proof of Theorem
\ref{3.22.1} we have that $u\in\WO^{2}_{p}(\bR^{d}_{+})$
and estimate \eqref{eq12.53} holds (with $F(D^{2}u ,x )-\lambda u=f$).

Next, we consider the general situation.
Take the function $\zeta$ from the proof of assertion (i)
 of Theorem \ref{thm3.4}
but such that \eqref{7.13.4} holds with $p=1$. For $n=1,2,...$
define
$$
F_{n}(u'',x)=\int_{\bR^{d}_{+}}F (u'',x+y/n)\zeta(y)\,dy.
$$
Obviously, these infinitely differentiable functions of $x$
are positive homogeneous of degree one and satisfy
$(\rm H_{1})$ and Assumption \ref{assump2} with the same parameters
as $F$ does.

Then we approximate $f$ and $v$ in appropriate norms with functions
 $f_{n}$ and $v_{n}$ possessing the properties described
above. This yields a sequence of $u_{n}\in W^{2}_{p}(\bR^{d}_{+})$
with  uniformly  bounded $W^{2}_{p}(\bR^{d}_{+})$ norms and
such that $u_{n}-v_{n}\in  \WO^{2}_{p}(\bR^{d}_{+})$ and
$$
F_{n}(D^{2}u_{n},x)-\lambda u_{n}(x)=f_{n}(x).
$$
By embedding theorems there is a  $u \in W^{2}_{p}(\bR^{d}_{+})$
and a subsequence,which is still denoted by $\{u_{n}\}$,  such that
$u_{ n }\to u$ uniformly on compact subsets of
$\bar{\bR}^{d}_{+}$. In particular $u=v$ on $\partial\bR^{d}_{+}$,
so that $u -v \in  \WO^{2}_{p}(\bR^{d}_{+})$.

Since $f_{n}\to f$ in $L_{p}(\bR^{d}_{+})$, we may assume that
$f_{ n }\to f$ (a.e.). Therefore, if we define
$$
\bar{F}_{n_{0}}(u'',x)=\sup_{n\geq n_{0}}F_{n}(u'',x),\quad
\underline{F}_{n_{0}}(u'',x)=\inf_{n\geq n_{0}}F_{n}(u'',x),
$$
then, for any $n_{0}$, (a.e.)
$$
\nliminf_{n\to\infty}\bar{F}_{n_{0}}(D^{2}u_{n},x)
\geq
\nliminf_{n\to\infty}F_{n}(D^{2}u_{n},x)=f(x) + \lambda u(x),
$$
$$
\nlimsup_{n\to\infty}\underline{F}_{n_{0}}(D^{2}u_{n},x)
\leq
\nlimsup_{n\to\infty}F_{n}(D^{2}u_{n},x)=f(x)+\lambda u(x).
$$
It follows by Theorems 3.5.15 and 3.5.6
of \cite{Kr87} that for any $n_{0}$ (a.e.)
\begin{equation}
                                                     \label{8.5.1}
\bar{F}_{n_{0}}(D^{2}u ,x)\geq f(x)+\lambda u(x)
\geq \underline{F}_{n_{0}}(D^{2}u ,x).
\end{equation}

Now observe that, for each $u''\in\cS$,
$F_{n}(u'',x)\to F(u'',x)$  (a.e.). Since both parts are
positive homogeneous and
Lipschitz continuous in $u''$ with constant independent of $n$
we also have
 $$
 \bar{ F}_{n_{0}}(u'',x)- \underline{F}_{n_{0}}(u'',x)
\leq \varepsilon_{n_{0}}(x)|u''|,
 $$
where $\varepsilon_{n_{0}}(x)\to0$ (a.e.) as $n_{0}\to\infty$.
After that to finish proving the existence
it only remains to pass to the limit
in \eqref{8.5.1}.

Uniqueness is proved in the same way as in Theorem \ref{3.22.1}.
The theorem is proved.
\end{proof}

\mysection{Proof of Theorem \protect\ref{mainthm2}}
                    \label{SecproofThm2}

First we state a generalization of a result
of \cite{Kr10}. The point is that in that paper the counterpart
of our Assumption \ref{assump2} is formulated
as \eqref{8.5.2} (in its elliptic version).

\begin{theorem}
                                            \label{theorem 8.5.2}
Let $p>d$. Then there exist constants $\theta\in (0,1]$
depending  only on $d,p, \delta$   and $\lambda_0$ depending
only on $d,p, \delta$, and $R_0$
such that if Assumption \ref{assump2} holds with this $\theta$
and $\cD=\bR^{d}$, then

(i) For any $u\in W^{2}_p(\bR^{d} )$
 satisfying  \eqref{goal}, we have
\begin{equation}
                                               \label{8.5.3}
\lambda\|u\|_{L_p(\bR^{d} )}+\|D^2u\|_{L_p(\bR^{d} )}\leq
 N\| f \|_{L_p(\bR^{d} )},
\end{equation}
if
$\lambda\geq  \lambda_0 $,
where $N=N(d, p,\delta)$, and we have
\begin{equation}
                                \label{8.5.4}
\|u\|_{W^2_p(\bR^{d}_{+})} \leq N \|f\|_{W^2_p(\bR^{d}_{+})}
\end{equation}
if $\lambda>0$ with $N=N(d, p, \delta, R_0,\lambda)$.

(iii) For any   $\lambda>0$
and $f \in L_p(\bR^{d} )$,
there exists a
unique
$u\in W^{2}_p(\bR^{d} )$ satisfying   \eqref{goal}.
\end{theorem}

We only give a few comments on the proof of this theorem.
In case $F$ is independent of $x$ the a priori estimate
\eqref{8.5.3} is obtained in \cite{Kr10} for all $\lambda>0$.
When $F$ depends on $x$ one obtains the estimate \eqref{8.5.3} for
$\lambda
\geq\lambda_{0}$ and \eqref{8.5.4} for $\lambda>0$
as in the proof of Theorem \ref{thm3.4}. After the necessary
a priori estimates are obtained, it is stated
in  \cite{Kr10} that the solvability theorems are
derived in a standard way without giving any specific
details. This standard way is
presented in the proof of Theorem \ref{theorem 8.5.1}.

\begin{proof}[Proof of Theorem \ref{mainthm2}]
As in the proof of Theorem \ref{thm3.4}, we first establish
\eqref{eq00.34} as an a priori estimate and
the case of general $g$ is reduced to the case $g\equiv 0$ by replacing
the unknown function $u$ with $u-g$.
We will see that to obtain the a priori estimate
we do not need    condition  (H$_3$).

Observe that Theorems \ref{thm3.4} and \ref{theorem 8.5.2}
with $\lambda=\lambda_{0}$ imply that
$$
\|D^2u\|_{L_p(\bR^{d}_{+})}\leq
 N(\|F(D^{2}u) \|_{L_p(\bR^{d}_{+})}+\| u \|_{L_p(\bR^{d}_{+})}),
\quad   \forall u\in \WO^{2}_p(\bR^{d}_{+}) ,
$$
\begin{equation}
                                                        \label{8.10.1}
\|D^2v\|_{L_p(\bR^{d}) }\leq
 N(\|F(D^{2}v) \|_{L_p(\bR^{d} )}+\| v\|_{L_p(\bR^{d} )}),
\quad   \forall v\in W^{2}_p(\bR^{d} ),
\end{equation}
where $N=N(d, p, \delta,R_{0})$ (provided that $\theta
=\theta(d,p,\delta)$ is chosen appropriately).

Now assume that $u\in\WO^{2}_{p}(\cD)$ satisfies
\begin{equation}
                                                \label{7.14.1}
 F(D^{2}u( x), x)+G(D^{2}u( x),D u( x),u( x), x) =0.
\end{equation}
We move the term $G$ to the right-hand side and
define
$$
f(x):=-G(D^{2}u( x),D u( x),u( x), x).
$$
After that
by using the technique  based on flattening the boundary,
partitions of unity, and interpolation inequalities
allowing one to estimate $Du$ through $D^{2}u$ and $u$  and
also using
\eqref{8.10.1} we obtain that
\begin{equation}
                        \label{eq16.07}
 \|D^2 u\|_{L_p(\cD)}
\le N_1(\| f \|_{L_p(\cD)}+\| u \|_{L_p(\cD)}),
\end{equation}
provided that $\theta$ is sufficiently small
 depending only on $d$, $p$,
$\delta$, and the $C^{1,1}$ norm of $\partial D$.
Here $N_1$ depends only on $d$,
$p$, $\delta$,  $R_0$, and the  $C^{1,1}$  norm of $\partial \cD$.
Below by $N_{i}$ we denote the same type of constants as $N_{1}$.
It follows from the definition of $f$ and $({\rm H}_2)$ that,
for any $s>0$,
$$
\|f\|_{L_p(\cD)}\le \|\chi(|D^2 u|)D^2 u\|_{L_p(\cD)}+
K(\|Du\|_{L_p(\cD)}+\|u\|_{L_p(\cD)})
$$
$$
+\|\bar{G}\|_{L_p(\cD)}\le \chi(s)\|D^2
u\|_{L_p(\cD)}+\|\chi\|_{L_\infty}s|\cD|^{1/p}
$$
\begin{equation}
                            \label{eq23.45}
+
K(\|Du\|_{L_p(\cD)}+\|u\|_{L_p(\cD)})+\|\bar{G}\|_{L_p(\cD)}.
\end{equation}
Upon taking $s$ large so that $N_1\chi(s)\le 1/2$,
we get from \eqref{eq16.07}, \eqref{eq23.45}, and the
interpolation inequality that
\begin{equation}
                        \label{eq00.31}
\|u\|_{W^2_p(\cD)}
\le  N_2(\|u\|_{L_p(\cD)}+\|\bar G\|_{L_p(\cD)}+
\|\chi\|_{L_\infty}s|\cD|^{1/p}).
\end{equation}

To estimate the term $ \|u\|_{L_p(\cD)}$ on the right-hand
side of \eqref{eq00.31},
we rewrite \eqref{7.14.1} as
$$
F(D^{2}u( x), x)+G(D^{2}u( x),D u( x),u( x), x)
$$
\begin{equation}
                            \label{eq09.30}
-G(D^{2}u( x),0,0, x) =-G(D^{2}u(x),0,0,x).
\end{equation}
Note that, by conditions $({\rm H}_1)$ and $({\rm H}_2)$, there
exist $L\in \bL_\delta$  and bounded measurable functions
$b=(b^1,...,b^d)$ and $c$ such that the left-hand side of \eqref{eq09.30}
can be represented as $Lu+b^iD_i u-cu $.  Since
 $G$ is nonincreasing in $u$, we have $c\ge 0$.
Therefore, by Alexandrov's estimate
$$
\sup_{\cD}|u|,\quad
 \|u\|_{L_p(\cD)}\le N \|G(D^{2}u( \cdot),0,0, \cdot)\|_{L_p(\cD)},
$$
where $N=N(d,p,\delta,{\rm diam} (\cD) )$.
Again by condition $({\rm H}_2)$, for any $s_1>0$,
$$
 \|u\|_{L_p(\cD)}\le N\|\chi(|D^2 u|)D^2 u\|_{L_p(\cD)}+
N\|\bar{G}\|_{L_p(\cD)}
$$
\begin{equation}
                            \label{eq09.43}
\le N_{3}(\chi(s_1)\|D^2 u\|_{L_p(\cD)}+ \|\chi\|_{L_\infty}
s_1|\cD|^{1/p}+ \|\bar{G}\|_{L_p(\cD)}).
\end{equation}
Combining \eqref{eq00.31} with \eqref{eq09.43} and taking
$s_1$ sufficiently large so that
$ N_2N_{3} \chi(s_1)\le 1/2$, we arrive at
$$
\|u\|_{W^2_p(\cD)}
\le N_4\|\bar G\|_{L_p(\cD)}+N_4(s+s_1) \|\chi\|_{L_\infty}|\cD|^{1/p},
$$
which is \eqref{eq00.34} in the case that $g=0$.

To prove the existence and uniqueness of solutions, we first assume
that $H:=F+G$ and $g$ are
smooth in $x$ and the domain is of class
$C^{2+\alpha}$. Then, under conditions $({\rm H}_2)$ and $({\rm H}_3)$,
it is known (cf. the proof of Theorem \ref{theorem 8.5.1}) that there
is a unique
$C^{2}(\cD)$ solution $u$   with a given smooth boundary data.
This solution is certainly in $W^{2}_{p}(\cD)$ and we have
an estimate of its $W^{2}_{p}(\cD)$ norm. After that
we mollify the original $F$ and $G$ in $x$, mollify $g$,
and approximate $\cD$  by a sequence of increasing  smooth
domains  $\cD_n$  with the $C^{1,1}$ norm under control. We take
these domains because  otherwise
after mollifications $F$ may fail
to satisfy \eqref{7.28.1} for all $x\in\cD$.
After that it suffices to repeat the last part of the proof
of Theorem \ref{theorem 8.5.1}.
To see that the limiting function $u$ satisfies $u-g\in \WO^2_p(\cD)$,
we use the fact that $u_n,g_n\in W^2_p(\cD_n)$ with uniformly bounded norms and the fact that $(u_n-g_n)I_{\cD_n}\in \WO^1_p(\cD)$ with uniformly bounded norms.
Of course, while passing to the limits
and proving uniqueness we use that
for any $u,v\in W^{2}_{p}(\cD)$ there is an operator
$L\in\bL_{\delta}$ and  bounded measurable functions $b=(b^1,...,b^d)$
and $c$ satisfying $|b|\le K,\quad 0\le c\le K$
 such that
 $$
H(D^{2}u,Du,u,x)-H(D^{2} v,Dv,v,x)
$$
$$
=H(D^{2}u,Du,u,x)-H(D^{2}v,Du,u,x)
$$
$$
+H(D^{2}v,Du,u,x)-H(D^{2}v,Dv,v,x)
$$
$$
=L(u-v)+b^{i}D_{i}(u-v)-c(u-v).
$$
The theorem is proved.
\end{proof}

\mysection{Parabolic Bellman's  equations in $\bR^{d+1}$
with constant coefficients}
                                \label{sec4}

In this section we consider the equation
\begin{equation}
                    \label{para0}
\partial_t u(t,x) + F(D^2u)= f(t, x),
\end{equation}
in the whole space. Due to the same reasons as in Section \ref{sec2},
equation
\eqref{para0} can be written as a parabolic Bellman's equation
$$
\partial_t u(t,x) + \sup_{\omega\in \Omega} [a^{ij}(\omega)
 D_{ij} u(t,x) ]= f(t, x).
$$

For $r>0$, introduce $Q_r:=Q_r(0,0)$.
The following is Lemma 4.2.2 in \cite{Kr08}.
\begin{lemma}  \label{3.7.1}
Let $p\in [1,\infty)$. Then there is a constant $N=N(d,p)$ such that for any $r\in(0,\infty)$ and $u\in C^{\infty}_{loc}(\bR^{d+1})$ we have
$$
\dashint_{Q_{r}} |Du -(Du)_{Q_{r}}|^p \,dx\,dt
\leq N r^p\dashint_{Q_{r}} (|D^{2}u|+|\partial_{t}u|)^p \,dx\,dt,
$$
$$
\dashint_{Q_{r}} |u(t,x)- (u)_{ Q_{r}}- x^{i}
 (D_{i}u)_{Q_{r}}|^p \,dx\,dt
$$
$$
\leq Nr^{2p} \dashint_{Q_{r}}(|D^{2}u|+|\partial_{t}u|)^p\,dx\,dt.
$$
\end{lemma}

The second lemma  is a  parabolic embedding theorem
proved as Lemma II.3.3 in \cite{LSU}.
\begin{lemma} \label{3.7.2}
Let $p> (d+2)/2$ and $u \in W^{1,2}_p(Q_{1} )$.  Then for any
$(t,x)\in  Q_{1} $,
$$
|u(t,x)|\leq N\|u\|_{W^{1,2}_p(Q_{1})},
$$
where $N=N(d, p)$.
\end{lemma}

Let $v= u(t,x)- (u)_{ Q_{r} }- x^{i} (D_{i}u)_{  Q_{r} }$,
then $v$ belongs to $W^{1,2}_p (  Q_{r} )$
whenever $u$ does.
Noting that
$$
D_{ij}v= D_{ij}u, \quad \partial_{t}u =\partial_{t}v
, \quad D_{i}v=  D_{i}u -
( D_{i}u )_{  Q_{r}},
$$
we get the following corollary by dilations and combining Lemmas \ref{3.7.1} and
\ref{3.7.2}.
\begin{corollary}
                          \label{3.7.2b}
Let $p>(d+2)/2$ and $r\in(0,\infty)$. Then for any $u\in
W^{1,2}_p(Q_{r})$, we have
$$
\sup_{(t,x)\in  Q_{r}}| u(t,x)- (u)_{  Q_{r} }-
x^{i} (D_{i}u)_{  Q_{r} }|^{p}
$$
$$
\leq Nr^{2p} \dashint_{ Q_{r}}(|D^{2}u |+ |\partial_{t}u|)^p \,dx\,dt,
$$
where $N=N(d,p)$.
\end{corollary}

\begin{lemma}
                        \label{3.7.3}
Let $r\in(0, \infty)$  and $\kappa\geq2$.
Let
$
v\in C^{1,2}_b ( Q_{\kappa r} )
$
be a solution of \eqref{para0} with $f\equiv0$.
Then there are constants $\alpha\in(0,1)$ and $N$ depending only
on $d$ and $\delta$ such that
$$
\dashint_{Q_r}\dashint_{Q_r} |D^2v(t, x)-D^2v(s, y)| \,dx\,dt\,dy\,ds
\leq N\kappa^{-2-\alpha} r^{-2}
\sup_{ \partial'Q_{\kappa r}} |v|.
$$
\end{lemma}
\begin{proof}
Dilations show that we may concentrate on the case when $r=1/\kappa$.
In this case one routinely derives from Theorem 5.5.2 in  \cite{Kr87}  that
there exist $\alpha, N$
depending only on $\delta$, $d$  such that  for any $(t,x), (s,y)\in Q_{1/2}$,
we have
\begin{equation}
                                                        \label{7.14.5}
|D^2 v(t, x)- D^2 v(s, y)| \leq N(|x-y|^{\alpha}+|t-s|^{\alpha/2})
\sup_{Q_{1}} |v|
\end{equation}
 Thanks to the maximum principle, the lemma is proved.
\end{proof}

\begin{remark}
                                            \label{remark 7.14.1}
By ``routinely derives" we mean the following. First
observe that we may assume that
$$
\sup_{Q_{1}}|v|=1.
$$
Indeed, if the sup is zero,   we have nothing to prove.
However if the sup is different
from zero we can replace $v$ with the ratio  of $v$ and the sup.

Then
we approximate  $F(u'')$
by smooth convex functions $F^{n}(u'')$ so that
 $F^{n}\to F$ as $n\to\infty$
uniformly on compact sets
and for all values of variables
$$
\delta|\xi|^{2}\leq F^{n}_{u_{ij}}\xi^{i}\xi^{j}\leq\delta^{-1}|\xi|^{2},
\quad
|F^{n}-F^{n}_{u_{ij}}u_{ij}|\leq 1.
$$
To do that it suffices to mollify $F(u'')$ with respect to
$u''$.
Then we approximate $ v$ on $\partial'Q_{1}$
uniformly by infinitely differentiable functions $\phi^{n}$
such that $|\phi^{n}|\leq1$. Next,
we apply Theorem 6.2.5  of  \cite{Kr87}  to find a unique $u^{n}\in
C^{1,2}(Q_{1})\cap C(\bar Q_{1})$
such that
$$
\partial_{t}u^{n}+F^{n}(D^{2} u^{n})-\tfrac{1}{n}u^{n}=0
\quad\text{in}\quad
Q_{1}
$$
and $u^{n}=\phi^{n}$ on $\partial'Q_{1}$.

This theorem also guarantees that
$$
u^{n},Du^{n},D^{2}u^{n},\partial_{t}u^{n}
\in C^{1,2} ([ \varepsilon, 1-\varepsilon ]\times \bar{B}_{\varepsilon})
$$
for any $\varepsilon\in(0,1/2)$.
By the maximum principle $u^{n}$ are uniformly bounded in
$Q_{1}$. Since
$$
\partial_{t}v+F^{n}( D^{2}v)-\tfrac{1}{n}v
= F^{n}(D^{2}v)-F(D^{2}v)-\tfrac{1}{n}v
$$
and the latter tends to zero uniformly in $Q_{1}$, by the maximum principle
$u^{n}\to v$ uniformly in $Q_{1}$.

Now we can formally apply Theorem
 5.5.2 of  \cite{Kr87}  and get that the
$C^{1+\alpha/2,2+\alpha}(Q_{1/2})$
norms of $u^{n}$ are uniformly bounded and, in particular,
for any $(t,x), (s,y)\in Q_{1/2}$,
we have
$$
|D^2 u^{n}(t, x)- D^2 u^{n}(s, y)| \leq N(|x-y|^{\alpha}+|t-s|^{\alpha/2}),
$$
where $N$ depends only on $d$ and $\delta$. Since $u^{n}\to v$
uniformly and $D^{2}u^{n}$ are uniformly  equicontinuous in
$Q_{1/2}$, we have that $D^{2}u^{n}\to  D^{ 2}v$
in $Q_{1/2}$, which
yields
$$
|D^2 v(t, x)- D^2 v(s, y)| \leq N(|x-y|^{\alpha}+|t-s|^{\alpha/2})
$$
and this coincides with \eqref{7.14.5}.

\end{remark}

Introduce $\bL_{\delta}$ as before Lemma \ref{3.6.5} but allow the dependence
of the coefficients
on $(t,x)$ rather than on $x$ only.

\begin{lemma}
                       \label{plin}
Let $r\in(0,\infty)$ and let
$u\in   C(\bar Q_{r})\cap  W^{1,2}_{d+1}(Q_{\rho})$
for any $\rho\in (0,r)$. Then there are constants $\gamma\in(0, 1]$ and $N$,
depending only on
$\delta, d$, such that
for any $L\in  \bL_\delta $ we have
$$
\dashint_{Q_{r}} |D^2u|^\gamma \, dx \,  dt
 \leq  Nr^{-2\gamma}\sup_{\partial'Q_{r}}
 |u|^\gamma
$$
\begin{equation}
                                                                  \label{8.11.1}
+ N
\left(\dashint_{Q_{r}}|\partial_{t}u +Lu|^{d+1} \,
dx\, dt\right)^{\gamma/{(d+1)}}.
\end{equation}
\end{lemma}

\begin{proof} If we prove \eqref{8.11.1}
with $\rho$ in place of $r$ for any $\rho\in(0,r)$,
then by passing to the limit we will obtain \eqref{8.11.1} as is.
Hence, we may assume that $u\in W^{1,2}_{d+1}(Q_{r})$.
 As usual, we may also assume that $r=1$. Then
we may also assume that the coefficients $a^{ij}(t,x)$ of $L$
are infinitely differentiable in $\bR^{d+1}$. Now
set $f=\partial_{t}u+Lu$ in $Q_{1}$ and extend $f(t,x)$
for $t\leq 0$ as zero. Also set $u(t,x)=u(-t,x)$
for $t\leq0$. Observe that the new $u$ belongs to
$W^{1,2}_{d+1 }((-1,1)\times B_{1})$.
After that
define $v(t,x)$ as a unique
 $W^{1,2}_{d+1 }((-1,1)
\times B_{1})\cap C([-1,1]\times \bar B_1)$
solution of
 $\partial_{t}v+Lv=f$ with  terminal
and lateral conditions being $u$. The existence and uniqueness
of such a solution is a classical result
(see, for instance,
Theorem IV.9.1 of \cite{LSU} or
 Theorem 7.17 of \cite{GML}).
By uniqueness $v=u$ in $Q_{1}$, so that
owing to Corollary 4.2 of \cite{Kr10},
$$
\int_{ Q_1 } |D^2u|^\gamma \, dx\, dt
=\int_{ Q_1 } |D^2 v|^\gamma \, dx\, dt\leq N
\left(\int_{(-1,1)\times B_{1}}|f|^{d+1}\,dx\,dt
\right)^{\gamma/(d+1)}
$$
$$
+N\sup_{\partial'(-1,1)\times B_{1}} |v|^\gamma
=N
\left(\int_{Q_{1}}|f|^{d+1}\,dx\,dt
\right)^{\gamma/(d+1)}
+N\sup_{\partial'Q_{1}} |u|^\gamma .
$$
The lemma is proved.
\end{proof}

 We note that a slightly weaker statement than Lemma \ref{plin}
can be found in \cite{Wa90}, where
for the proof the reader  is referred to \cite{Wa92}.

Everywhere below in this section
$\alpha$ is the constant from Lemma \ref{3.7.3}
and $\gamma$ is the one from Lemma \ref{plin}.
\begin{lemma}
                        \label{3.8.1}
Let $ r\in(0,\infty)$ and  $\kappa \geq 2$. Let  $u\in
W^{1,2}_{d+1}(Q_{\kappa r})$ be a solution to  \eqref{para0}. Then
$$
\dashint_{Q_r}\dashint_{Q_r} |D^2u(t,x)-D^2u(s,y)|^{\gamma}\,
 dx\, dt\, dy\,ds
$$
$$
\leq N\kappa^{d+2}
( |f|^{d+1} )^{\gamma/(d+1)}_{Q_{\kappa r}}
+N\kappa^{-\alpha \gamma} ( |D^2u|^{d+1})^{\gamma/(d+1)}_{Q_{\kappa r}},
$$
where $N$ depends only on $d$ and $\delta$.
\end{lemma}
\begin{proof} As usual,
it suffices to prove the lemma for $r=1$.
We  follow  the proof  of Lemma 2.4 in \cite{Kr10}
and, as there, without trouble reduce the general case to the one that
$u\in C^{\infty}_{b}(\bar{Q}_{\kappa  })$.
Define $\hat{u}:= u-(u)_{Q_{\kappa  }}- x^i(D_{i}u)_{Q_{\kappa }}$ and
let $v\in C^{1,2}_b(Q_{\kappa  })\cap C (\bar Q_{\kappa  })$
be a solution of \eqref{para0} in $Q_{\kappa}$ with
$f\equiv0$
and
$v=\hat{u} $ on
$\partial' Q_{\kappa  }$. Such a solution $v$ exists by Theorem 6.4.1
of \cite{Kr87}.
By Lemma \ref{3.7.3}, H\"older's inequality, and Corollary
\ref{3.7.2b}, we have
$$
\dashint_{Q_1}\dashint_{Q_1}|D^2v(t,x)-D^2v(s,y)|^\gamma \,dx\, dt\,  dy\, ds
$$
\begin{equation}
                       \label{pv}
\leq N\kappa^{-\gamma(2+\alpha)}\sup_{\partial'Q_{\kappa}}
|v|^{\gamma}
\leq N \kappa^{-\alpha\gamma}
( |D^2u|^{d+1}+|\partial_{t}u| ^{d+1})^{\gamma/(d+1)}_{Q_{\kappa}}.
\end{equation}
Let $w:=\hat{u}-v$ in $\bar Q_{\kappa}$. Then by the  same argument as
in the proof of Lemma 2.4 in \cite{Kr10} or our Lemma \ref{osc},
  we obtain that there exists
an operator $L\in  \bL_\delta $,
such that $\partial_{t}w+Lw=f$.
Then by Lemma \ref{plin},
$$
\dashint_{Q_{1}} |D^2w|^\gamma \,dx\, dt\leq N\kappa^{d+2}
\dashint_{Q_{\kappa}} |D^2w|^{\gamma}\, dx \, dt
$$
$$
\leq N\kappa^{d+2} \left(\dashint_{Q_{\kappa}}
|f|^{d+1}\,dx\, dt\right)^{\gamma/ (d+1) }
$$
and
$$
\dashint_{Q_1}\dashint_{Q_1}|D^2w(t,x) -D^2 w(s,y)|^\gamma
\leq N\kappa^{d+2}\left(\dashint_{Q_{\kappa}}
|f|^{d+1}\,dx\, dt\right)^{\gamma/ (d+1) }.
$$
By combining this inequality and \eqref{pv} and observing that $D^{2}u=
D^{2}v+D^{2}w$  and
\begin{equation}
                            \label{eq16.29}
|\partial_t u|=| f-F(D^2u)| \leq |f|+ N |D^2u|,
\end{equation}
 we get the desired result. The lemma is proved.
\end{proof}

 The next theorem is the main result of this section.
For simplicity of notation set
\begin{equation}
                                                         \label{7.10.6}
L_{p}=L_p(\bR^{d+1}),\quad W^{1,2}_{p}=W^{1,2}_p(\bR^{d+1}).
\end{equation}

\begin{theorem}
                            \label{thm4.7}
Let $p>{d+1}$.
(i)
Let $u\in W^{1,2}_{p} $
be a solution to  \eqref{para0}. Then
\begin{equation}
                     \label{3.9.2}
\|D^2u\|_{L_p }
+\|\partial_{t}u\|_{L_p }\leq N \|f\|_{L_p },
\end{equation}
where $N$ depends only on $p$, $d$, and $\delta$.

(ii) For any $\lambda>0$  and
$f\in L_{p} $, there exists a unique solution
$u\in W^{1,2}_{p} $ of the equation
\begin{equation}
                        \label{eq12.34}
\partial_t u + F(D^2u)-\lambda u= f.
\end{equation}
Furthermore,
\begin{equation}
                                                       \label{7.6.1}
\lambda\|u\|_{L_p } +\|D^2u\|_{L_p }
+\|\partial_{t}u\|_{L_p }\leq N \|f\|_{L_p },
\end{equation}
 where $N$ depends only on $p$, $d$, $\delta$, and $\lambda$.
\end{theorem}

\begin{proof}
(i)  The estimate of the $D^2u$ term on the left-hand side of
\eqref{3.9.2} is derived from Theorem 5.3 of  \cite{Kr10} and Lemma
\ref{3.8.1} in the same way as Theorem 2.5 (i) of
 \cite{Kr10} or Theorem \ref{3.22.1} (i).
Of course this time we use the filtration
of parabolic dyadic cubes.
The estimate of $\partial_t u$ follows from that of $D^2 u$ and \eqref{eq16.29}.

(ii)   To prove the a priori
estimate \eqref{7.6.1} we replace $f$ with $-\lambda {u}+f$ in the above
estimates  and get
$$
\|\partial_{t}u\|_{L_p }+\|D^2u\|_{L_p }
 \leq \lambda \|u\|_{L_p }+\|f\|_{L_p }.
$$
Hence it suffices  to prove that
$$
\lambda\|u\|_{L_p }\leq N\|f\|_{L_p },
$$
which is done in the same way as in the elliptic case.
After that, the solvability of \eqref{eq12.34}
is proved in the same way as in Theorem \ref{3.22.1}.
The theorem is proved.
\end{proof}

\mysection{Parabolic equations in $\bR^{d+1}$
with VMO coefficients}
                                \label{sec5}

In this section, we consider the parabolic equation
\begin{equation}
                            \label{vpara}
\partial_t u(t,x)+F(D^2u(t,x),t,x) -\lambda u(t,x)=f(t,x).
\end{equation}

Everywhere below in this section, Assumption \ref{assump1} is supposed to
hold with $\cD=\bR^{d+1}$, $\alpha$ is the constant from Lemma
\ref{3.7.3}, and $\gamma$ is the one from Lemma \ref{plin}.
We use notation \eqref{7.10.6} and recall that $\theta(\mu,d ,\delta
)$ is  introduced in Remark~\ref{remark 7.29.1}.

\begin{lemma}
                    \label{pposc}
Let  $\beta\in(1,\infty)$, $ \lambda=0$, $\mu,r\in(0,
\infty)$,
$\kappa\geq 2$, and $ (t_0,x_0)\in
\bR^{d+1}$.  Suppose that $\theta=\theta(\mu,d ,\delta )$. Let $u\in
W^{1,2}_{d+1} $ be a solution
of \eqref{vpara} vanishing outside
$Q_{R_0}(t_0,x_0)$. Then,
$$
\dashint_{Q_r}\dashint_{Q_r} |D^2u(t,x)-D^2u(s,y)|^\gamma
 \,dx\, dt \,dy\,ds
\leq N\kappa^{d+2}\left(|f|^d\right)_{Q_{\kappa
r}}^{\gamma/{(d+1)}}
$$
$$
+N\kappa^{d+2}\left(|D^2u|^{\beta(d+1)}
\right)_{Q_{\kappa r}}^{\gamma/ (\beta d+\beta)}
\mu^{\gamma/ (\beta' d+\beta') }
$$
\begin{equation}
                         \label{ppposc}
+N\kappa^{-\alpha\gamma}\left( |D^2u|^{d+1} \right)^{\gamma/(d+1)}_{Q_{\kappa r}},
\end{equation}
where $N=N(d,\delta ,\beta )$ and $\beta'=\beta/(\beta-1)$.
\end{lemma}

\begin{proof}
We will basically repeat the proof of
Lemma \ref{osca} adapting it to the parabolic case
and the whole space.
Introduce
$$
\bar{F}(u'')=\left\{
               \begin{array}{ll}
                 (F)_{Q_{R_0}(t_0,x_0)}(u''), & \hbox{if $\kappa r\geq R_0$;} \\
                 (F)_{Q_{\kappa r}}(u''), & \hbox{otherwise,}
               \end{array}
             \right.
$$
and
$$
h(t,x)=\sup_{u''\in \cS: |u''|=1}|F(u'',t,x)-\bar F(u'')|.
$$
Note that
$$
\partial_{t}u +\bar F(D^2 u) =
\tilde f,
$$
where
$$
\tilde{f}(t,x)= f(t,x)+\bar F(D^2 u)-F(D^2 u,t,x).
$$
By Lemma \ref{3.8.1} and the triangle inequality,
$$
\dashint_{Q_r}\dashint_{Q_r} |D^2u(t,x)-D^2u(s,y)|^\gamma \,dx\, dt \,dy\,ds
$$
$$
\le N\kappa^{d+2}\left((
|f|^{d+1})^{\gamma/(d+1)}_{Q_{\kappa r}}
+J^{\gamma/(d+1)}\right)
$$
\begin{equation}
               \label{3.10.1}
+ N\kappa^{-\alpha\gamma}
\left( |D^2u|^{d+1} \right)^{\gamma/(d+1)}_{Q_{\kappa r}},
\end{equation}
where $N=N(d,\delta)$ and
$$
J=\dashint_{Q_{\kappa r}}  | \bar F(D^2 u)-F(D^2 u,t,x)|^{d+1}I_{Q_{R_0}}(t_0,x_0) \, dx\, dt
\leq J_1^{1/\beta}J_2^{1/\beta'},
$$
with
$$
 J_1 =\dashint_{Q_{\kappa r}}  |D^2u|^{\beta(d+1)} \,dx\,dt,
$$
$$
J_2=\dashint_{Q_{\kappa r}}  h^{\beta'(d+1)} I_{Q_{R_0}(t_0,x_0)}\,dx\,dt
\leq N \dashint_{Q_{\kappa r}} h I_{Q_{R_0}(t_0,x_0)} \,dx \,dt.
$$
If $\kappa r < R_0$, we have
$$
J_2\leq N \dashint_{Q_{\kappa r}} h \,dx\,dt \leq N\theta.
$$
If $\kappa r\geq R_0$, we have
$$
J_2 \le N(\kappa r)^{-d-2}\int_{Q_{R_0}(t_0,x_0)}h\,dx\,dt
$$
$$
\le N(\kappa r)^{-d-2}R_0^{d+2}\dashint_{Q_{R_0}(t_0,x_0)}h\,dx\,dt
\le N\mu.
$$
Therefore, in any case,
$$
J\leq
 N\left(\dashint_{Q_{\kappa r}}
| D^2u(x)|^{\beta(d+1)}\,dx\,dt\right)^{1/\beta}\mu^{
1/\beta'}.
$$
Substituting the above inequality back into \eqref{3.10.1}, we
get \eqref{ppposc}. The lemma is proved.
\end{proof}

From Lemma \ref{pposc}, by a standard argument using
Theorem 5.3 of \cite{Kr10}
and the Hardy--Littlewood theorem, we arrive at the following corollary.

\begin{corollary}
                         \label{3.16.1}
Let $p>d+1$, and  $u\in W^{1,2}_{d+1} $ be a
solution of \eqref{vpara} with $\lambda=0$ vanishing outside $Q_{R_0}$.
Then there exist constants N and $\theta$ depending only on $p$, $d$, and
$\delta$, such that if Assumption \ref{assump1} is satisfied with this
$\theta$, then
$$
\|D^2 u\|_{L_p }+ \|\partial_{t}u
\|_{L_p } \leq N\|f\|_{L_p }.
$$
\end{corollary}

For any $T\in [-\infty,\infty)$, we denote
$$
\bR^{d+1}_T= \{(t,x)\in \bR^{d+1}: t>T\}.
$$
The main result of this section is the following theorem.

\begin{theorem}
                           \label{7.13.2}
Let $p>d+1$ and $T\in [-\infty,\infty)$. Then there exist
$\theta\in (0,1]$,
 depending only on $d,\delta,p$
 and a constant $\lambda_0$, depending only on $d,\delta,p$,
 and $R_0$,
such that if Assumption \ref{assump1} is satisfied with this
$\theta$, then

(i) For any $\lambda\ge \lambda_0$ and any $u\in W^{1,2}_{d+1}(\bR^{d+1}_T) $ satisfying
\eqref{vpara}, we have
\begin{equation}
                                   \label{3.17.1}
\lambda\|u\|_{L_p(\bR^{d+1}_T)}
+\|\partial_{t}u\|_{L_p(\bR^{d+1}_T)}+
\|D^2 u\|_{L_p(\bR^{d+1}_T)}
\leq N\|f\|_{L_p(\bR^{d+1}_T)},
\end{equation}
where $N=N(d,\delta,p)$

(ii) For any $\lambda>0$, there exists a constant
$N=N(d,p, \delta, R_0,\lambda)$ such that for any $u\in
W^{1,2}_p(\bR^{d+1}_T)$ satisfying \eqref{vpara} we have
\begin{equation}
                                   \label{eq18.26}
\|u\|_{W^{1,2}_p(\bR^{d+1}_T)} \leq N\|f\|_{L_p(\bR^{d+1}_T)}.
\end{equation}

(iii) For any $\lambda>0$ and $f\in L_p(\bR^{d+1}_{T})$, there exists a unique solution of \eqref{vpara} in $W^{1,2}_p(\bR^{d+1}_T)$.
\end{theorem}

\begin{proof}
First we assume $T=-\infty$.
The proof of Theorem  \ref{theorem 8.5.1} shows that
assertion (iii) follows from (i) and (ii).
We suppose that Assumption \ref{assump1} holds with $\theta$ from
Corollary \ref{3.16.1}.

Take a nonnegative function $\zeta\in C^{\infty}$ which
 has support in $-Q_{R_0}$ and is such that $\zeta^p$ integrates to one.
Fix $(s,y)\in\bR^{d+1}$, and define
$$
 u_{(s,y)} (t,x)=u(t,x)\zeta(s-t,y-x),
$$
Then $u_{(s,y)}(t,x)$ is supported in $Q_{R_0}(s,y)$, and
$$
\partial_{t}
u_{(s,y)}+F(u_{(s,y)},t,x)=f_{(s,y)},
$$
where
$$
f_{(s,y)}(t,x)=f(t,x)\zeta(s-t,y-x)+F(u_{(s,y)},t,x)
$$
$$
-F(\zeta(s-t,y-x)D^2u,t,x)
- (\partial_{t}\zeta)(s-t,y-x)u+\lambda u_{(s,y)}.
$$
By Corollary \ref{3.16.1} and condition $({\rm H}_1)$,
$$
\|\zeta(s-\cdot,y-\cdot)\partial_{t}u \|_{L_p}^p +
\|\zeta(s-\cdot,y-\cdot)D^2 u\|_{L_p}^p
$$
$$
\leq N \|\zeta(s-\cdot,y-\cdot)f\|_{L_p}^p+N\|
|D\zeta(s-\cdot,y-\cdot)|Du\|_{L_p}^p
$$
$$
+\left\|\left(|\partial_{t}\zeta |
+|D^2\zeta|
+\lambda |\zeta| \right)
(s-\cdot,y-\cdot)
u\right\|_{L_p}^p.
$$
Integrating the above inequality over $(s,y)\in\bR^{d+1}$ we get
$$
\|\partial_{t}u \|^p_{L_p}+\|D^2u\|^p_{L_p}
\leq N_1(\|f\|_{L_p}^p+\lambda^p \|u\|_{L_p}^p)
$$
$$
+N_2(\|Du\|^p_{L_p}+\|u\|^p_{L_p}),
$$
where $N_1=N_1(d,\delta,p)$ and $N_2=N_2(d,\delta,p,R_0)$.
Now to conclude \eqref{3.17.1} and \eqref{eq18.26},
it suffices to use again the proof of Lemma 3.5.5 of \cite{Kr87} as in
Theorem \ref{thm3.4}. This completes the proof of the
 theorem in the
special case when $T=-\infty$.

For $T>-\infty$, we extend $f$ to be zero for $t\le T$,
and then find a unique solution $\tilde u\in W^{1,2}_p(\bR^{d+1})$ of
\eqref{vpara} in $\bR^{d+1}$, the existence of which is guaranteed by the
argument above. This in turn also yields the existence of a solution of
\eqref{vpara} in $\bR^{d+1}_T$
satisfying \eqref{3.17.1} or \eqref{eq18.26} as appropriate.
Its uniqueness in $ W^{1,2}_p(\bR^{d+1}_T)$ follows
as usual from the uniqueness for linear equations
(with measurable coefficients) and parabolic Alexandrov's estimates.
 The theorem is proved.
\end{proof}

We finish the section by proving the following result
about the Cauchy problem. Denote
by $\WO^{1,2}_p((0,T)\times
\bR^{d})$ the set of functions of class $W^{1,2}_p((0,T)\times
\bR^{d})$ having zero trace on the plane $\{(T,x):x\in\bR^{d}\}$.
\begin{theorem}
Let $p>d+1$ and $T>0$. Then there exists $\theta\in (0,1]$
depending only on $d,\delta,p$, such that if Assumption \ref{assump1} is
satisfied with this $\theta$, the following assertions hold:

(i) For any $v\in W^{1,2}_p((0,T)\times\bR^{d})$ and
$f\in L_p((0,T)\times\bR^{d})$, there exists a unique solution $u\in
W^{1,2}_{p}((0,T)\times \bR^{d})$ of \eqref{vpara} in $(0,T)\times
\bR^{d}$ with $\lambda=0$ satisfying $u-v\in \WO^{1,2}_p((0,T)\times
\bR^{d})$.

(ii) Moreover,
\begin{equation}
                        \label{eq17.49}
\|u\|_{W^{1,2}_p((0,T)\times \bR^{d})}\le N\|v\|_{W^{1,2}_p((0,T)
\times \bR^{d})}+N\|f\|_{L_p((0,T)\times \bR^{d})},
\end{equation}
where $N=N(d,\delta,p,T,R_{0})$.
\end{theorem}
\begin{proof}
 As in the proof of Theorem  \ref{theorem 8.5.1},
it suffices to prove \eqref{eq17.49} as an a priori estimate.
By considering $u-v$ instead of the unknown function $u$, without loss of
generality we may assume that $v\equiv 0$.
Furthermore, having in mind the possibility of substitution
$\hat u=e^tu $, we see that it suffices to consider equation
\eqref{assump1} with
$\lambda=1$.  We extend $u$ to be zero for $t>T$. It is easily seen that
the extended $u\in W^{1,2}_p(\bR^{d+1}_0)$
satisfies \eqref{vpara} in $\bR^{d+1}_0$ with $f(t,x)=0$ for
$t \ge T$. Estimate \eqref{eq17.49} then follows from Theorem \ref{7.13.2} (ii).
\end{proof}

\mysection{Parabolic Bellman's equations in
$\bR^{d+1}_+$ with constant coefficients}
                                        \label{sec6}

In this section, we consider equation \eqref{para0} in the half space
$$
\bR^{d+1}_+:=\bR\times\bR^{d}_+.
$$
 For $r>0$, $t\in \bR$ and $x=(x^1,x')\in \bR^d_+$, denote
$$
Q_r^+(t,x)=Q_r(t,x)\cap \bR^{d+1}_+,\quad Q_r^+= Q_r^+(0,0),
\quad Q_r^+(x^1)=Q^+_r(0,x^1,0).
$$

The following lemma can be deduced from Corollary \ref{3.7.2b} in the same way as Lemma \ref{3.6.2} is proved.
\begin{lemma}
                          \label{3.21.1}
Let $p>(d+2)/2$ and $r\in(0,\infty)$.  Then for any $u\in
 W^{1,2}_p(Q_r^+) $ vanishing on $x^1=0$, we have
$$
\sup_{(t,x)\in {Q_{r}^+}} |u- x^{1}(D_1u)_{Q^+_{r}
}|^{p}\leq Nr^{2p} \dashint_{Q_{r}^+}(|D^2u|+ |\partial_{t}u |)^p \,dx\,dt.
$$
where $N$ depends only on $d$ and $p$.
\end{lemma}

\begin{lemma}
                        \label{3.21.2}
Let $r\in(0, \infty)$, $\kappa\geq2$, and
$v\in C (\bar{Q}_{\kappa r}^+)
\cap C^{1,2}_{b}(Q^{+}_{\kappa\rho})$ for any $\rho\in(0,r)$. Assume that
$v$ is a solution of \eqref{para0} with
$f\equiv0$ and $v=0$ on $x^1=0$. Then there are constants
$\alpha\in(0,1)$ and $N$, depending only
on $d$ and $\delta$, such that
$$
[D^2 v]_{ C^\alpha(Q_r^+)}\leq N   (\kappa r)^{-2-\alpha}
 \sup_{\partial' Q^+_{\kappa r}}|v|.
$$
\end{lemma}
\begin{proof}
Dilations show that it suffices to prove the inequality for $\kappa r=1$.
 We take a smooth domain $\cD_1$ such
that $B_{3/4}^+\subset \cD_1\subset B_{1}^+$.
As in Lemma \ref{3.7.3}, it then follows
from Theorem 5.5.2 in \cite{Kr87} that
$$
[D^2 v]_{ C^\alpha (Q^+_{1/2})} \leq
N\sup_{(0,3/4)\times \cD_1} |v|.
$$
Owing to  the maximum principle, the lemma is proved.
\end{proof}

The next lemma is a consequence of Lemma \ref{plin} and
can be proved in the same way as Lemma \ref{3.6.5} is proved.
\begin{lemma}
                       \label{plin2}
Let $r\in(0,\infty)$ and let a function
$u\in C(\bar{Q}^{+}_{r})\cap
W^{1,2}_{d+1}(Q^+_{\rho})$ for any $\rho\in(0,r)$ and satisfy
$u=0$ on $\partial' Q^+_{r}$.
 Then there are constants
$\gamma\in(0, 1]$ and $N$, depending only on $\delta$ and $d$,  such that  for
any $L\in \bL_\delta$ we have
$$
\dashint_{Q^+_{r}} |D^2u|^\gamma \, dx \, dt
\leq N \left(\dashint_{Q^+_{r}}|\partial_t u+ Lu|^{d+1} \, dx \, dt\right)^{\gamma/{(d+1)}}.
$$
\end{lemma}

Everywhere below in this section
$\alpha$ is the smallest of the constants
called $\alpha$ in Lemmas \ref{3.7.3} and \ref{3.21.2}
and $\gamma$ is the smallest of the ones from Lemmas \ref{plin}
and \ref{plin2}.
\begin{lemma}
                        \label{lemma 3.21.3}
Let $ r\in(0,\infty)$, $\kappa \geq 16$, and
$x_0^1\ge 0$. Let
$u\in W^{1,2}_{d+1}(Q^+_{\kappa r}(x^1_0))$ be a solution to
\eqref{para0} in $ Q^+_{\kappa r}(x^1_0)$
vanishing on $Q^+_{\kappa r}(x^1_0)\cap \partial \bR^{d+1}$. Then
$$
\dashint_{Q^+_r(x^1_{0})}\dashint_{Q^+_r(x_0^{1})}
|D^2u(t,x)-D^2u(s,y)|^{\gamma}\, dx \,dt \,dy \, ds
$$
$$
 \leq N\kappa^{d+2} \left(\dashint_{Q^+_{\kappa r}(x_0^1)}
|f|^{d+1} \,dx\,dt\right)^{\gamma/(d+1)}
$$
\begin{equation}
                                \label{equation 7.14.1}
+N\kappa^{-\alpha \gamma}
\left(\dashint_{Q^+_{\kappa r}(x_0^1)}
 |D^2 u|^{d+1} \,dx\,dt\right)^{\gamma/(d+1)},
\end{equation}
where the constant $N$ depends only on $d$ and  $\delta$.
\end{lemma}

\begin{proof} As in the proof of Lemma \ref{osc}
due to
dilations, we
only need to consider the case $\kappa r=8$. Again, we consider the following two cases.

{\em Case 1: $x_0^1>1$.}  In this case, we have
$Q_{r\kappa /8}(x^1_0)\subset\bR^{d+1}_{+}$ and inequality
\eqref{equation 7.14.1} is an immediate consequence of
 Lemma \ref{3.8.1} since $\kappa/8\geq 2$.

{\em Case 2: $x_0^1\in[0,1].$} Since $r=8/\kappa\le 1/2$,
we have
$$
Q_r^+(x_0^1)\subset  Q_{ 3/2 }^+  \subset Q_{4}^+
\subset Q^+_{\kappa r}(x_0^1).
$$
By using a standard
approximating argument, we may
assume that $u\in C^\infty_b(\bar Q^{+}_{\kappa r}(x^1_0))$.
 Define $\hat{u}:= u-x^1(D_1u)_{Q_{4}^+}$.
 We claim that there exists
a function $v $ such that

(i) $v\in C(\bar{Q}_{4})$, $v=\hat{u}$ on $\partial' Q^+_{4}$;

(ii) $v\in C^{1,2}_{b}(\bar{Q}^{+}_{\rho})$ for any $\rho<4$;

(iii) $v$ satisfies \eqref{para0} in $Q^+_{4}$  with $f\equiv0$.

The proof of this claim is obtained as follows.
First we take smooth
domains $\cD_n$ such that
 $B_{4-1/n}^+\subset \cD_n\subset B_{4}^+$ set
$\cQ_{n}=(0,16)\times\cD_{n}$ and
by applying Theorem 6.4.1 of \cite{Kr87} find
unique $v_{n}\in C^{1,2}_{b}(\cQ_{n})\cap C(\bar{\cQ}_{n})$
satisfying  \eqref{para0} with $f\equiv0$, and boundary condition
$v_{n}=\hat{u}$ on $\partial'\cQ_{n}$. Then
one routinely derives from Theorem 5.5.2 in  \cite{Kr87}
(cf. Remark \ref{remark 7.14.1}) that
there exists $\beta\in(0,1)$ such that for any
$\rho<4$ the $C^{1+\beta/2,2+\beta}(Q^{+}_{\rho})$
norms of $v_{n}$ are bounded for all large $n$.
After that one takes a subsequence of $v_{n}$,
if necessary, and finds a function $v$ possessing
the above properties (ii) and (iii). That $v$ also
satisfies (i) is proved in the same way as a similar
statement is proved in Theorem 6.3.1 of \cite{Kr87}.

Now Lemmas \ref{3.21.2} and \ref{3.21.1}  and the maximum principle
easily yield that
$$
\dashint_{Q^+_r(x^1_0)}
\dashint_{Q^+_r(x^1_0)}|D^2v(t,x)-D^2v(s,y)|
\,dx\,dt \,dy \,ds
$$
 $$
\leq Nr^\alpha [D^2 v]_{ C^\alpha(Q_{ 3/2 }^+)}
\le N r^\alpha \sup_{\partial' Q^+_{4}} |v|
$$
$$
\leq N \kappa^{-\alpha}
\left(\dashint_{Q_4^+}(|D^2u|+|\partial_{t}u |)^{d+1}
\,dx\,dt\right)^{1/(d+1)}.
$$
Recall that $\gamma\in (0,1]$. By H\"older's inequality,
$$
\dashint_{Q^+_r(x^1_0)}\dashint_{Q^+_r(x^1_0)}
|D^2v(t,x)-D^2v(s,y)|^\gamma\,dx\,dt \,dy \,ds
$$
\begin{equation}
                             \label{equation 7.14.2}
\leq N \kappa^{-\alpha\gamma}\left(
\dashint_{Q_{\kappa r}^+(x_0^1)}(|D^2u|+
|\partial_{t}u |)^{d+1}\,dx\,dt\right)^{ \gamma /(d+1)}.
\end{equation}

Next for $w:=\hat{u}-v$ in $Q_{4}^+$, we have
 $w\in
W^{1,2}_{d+1}(Q^+_{\rho})
 $ for any $\rho<4$. By the
 same argument as in the proof of Lemma \ref{osc},  we know
that there exists  an operator $L\in
\bL_\delta$ such that $\partial_{t}w
+Lw=f$ in $Q_{4}^+$. By Lemma \ref{plin2} and the fact
 that $\kappa r=8$, we get
$$
\dashint_{Q_r^+(x_0^1)} |D^2w|^\gamma \,dx\, dt\leq
N\kappa^{d+2} \dashint_{Q_{4}^+}  |D^2w|^{\gamma}\, dx \, dt
$$
$$
\leq N\kappa^{d+2} \left(\dashint_{Q_{4}^+ }
|f|^{d+1}\,dx\, dt\right)^{\gamma/{(d+1)}}
$$
$$
\leq N\kappa^{d+2} \left(\dashint_{Q_{\kappa r}^+(x_0^1)}
|f|^{d+1}\,dx\, dt\right)^{\gamma/{(d+1)}}
$$
and
$$
\dashint_{Q_r^+(x_0^1)}\dashint_{Q_r^+(x_0^1)}
|D^2w(t,x) -D^2 w(s,y)|^\gamma \, dx\, dt \, dy\, ds
$$
$$
 \leq N\kappa^{d+2}\left(\dashint_{Q^+_{\kappa r}(x_0^1)}
 |f|^{d+1}\,dx\, dt\right)^{\gamma/{d+1}}.
$$
Upon  combining this inequality with \eqref{equation 7.14.2},
  observing that $D^2 u=D^2 v+D^2 w$, and
using \eqref{eq16.29} we get
 \eqref{equation 7.14.1}. The
lemma is proved.
\end{proof}

As in the proof  of Theorem \ref{3.22.1},
one derives the following theorem
from Lemma \ref{lemma 3.21.3}, the
Hardy--Littlewood theorem, and
Theorem 5.3 of \cite{Kr10}, which we apply to the
filtration of dyadic parabolic cubes belonging to $\bR^{d+1}_{+}$.
Denote by $\WO^{1,2}_{p}(\bR^{d+1}_+)$ the set of functions
from $W^{1,2}_{p}(\bR^{d+1}_+)$ with zero trace at $x^{1}=0$.
\begin{theorem}
                               \label{3.21.5}
Let $p>{d+1}$.
(i)
If $u\in \WO^{1,2}_{p}(\bR^{d+1}_+)$
satisfies \eqref{para0} in $\bR^{d+1}_+$,
then
$$
\|D^2u\|_{L_p(\bR_+^{d+1})}
+\|\partial_{t}u \|_{L_p(\bR_+^{d+1})} \leq N \|f\|_{L_p(\bR_+^{d+1})},
$$
where $N$ depends only on $d$, $\delta$, and $p$.

(ii) For any $f\in L_p(\bR_+^{d+1})$ and $\lambda>0$,
there exists a unique solution $u\in \WO^{1,2}_p(\bR^{d+1}_+)$ of
the equation
$$
\partial_t u(t,x) + F(D^2u(t,x))-\lambda u (t,x) = f(t, x).
$$
Furthermore,
\begin{equation*}
\lambda\|u\|_{L_p(\bR_+^{d+1})} +\|D^2u\|_{L_p(\bR_+^{d+1})}
+\|\partial_{t}u \|_{L_p(\bR_+^{d+1})}\leq N \|f\|_{L_p(\bR_+^{d+1})},
\end{equation*}
 where $N$ depends only on $d$, $\delta$, and $p$.
\end{theorem}

\mysection{Parabolic equations in $\bR^{d+1}_+$ with VMO coefficients}
                                        \label{sec7}

In this section, we consider parabolic equations
in $\bR^{d+1}_+$
 with variable coefficients
\begin{equation}
                            \label{vpara2}
\partial_t u(t,x)+F(D^2 u(t,x),t,x) -\lambda u(t,x)=f(t,x).
\end{equation}
In the sequel, Assumption \ref{assump1} is supposed to hold
with $\cD=\bR^{d+1}_+$,
and the constants $\alpha$ and $\gamma$ in Lemma \ref{pposc2}
are taken from Section~\ref{sec6}.
Recall that $\theta(\mu,d ,\delta )$
is introduced in Remark \ref{remark 7.29.1}.

\begin{lemma}
                    \label{pposc2}
Let $\beta\in(1,\infty)$, $\lambda=0$, $\mu,r>0$, $\kappa\ge
16$,
 $x_0^1\ge 0$, and $ (\tau,z)\in
\bR^{d+1}_+$. Suppose that $\theta=
\theta(\mu,d ,\delta )$. Let $u\in
\WO^{1,2}_{d+1}(\bR_+^{d+1})$ be a solution of \eqref{vpara2}
vanishing
outside $Q^+_{R_0}(\tau,z)$. Then
\begin{equation*}
\dashint_{Q_r^+(x_0^1)}\dashint_{Q_r^+(x_0^1)} |D^2u(t,x)
-D^2u(s,y)|^\gamma \,dx\, dt \,dy\,ds
\end{equation*}
$$
\leq N\kappa^{d+2}\left(\dashint_{Q^+_{\kappa r}(x_0^1)}
|f| ^d\,dx\,dt\right)^{\gamma/{(d+1)}}
$$
$$
+N\kappa^{d+2}\left(\dashint_{Q^+_{\kappa r}(x_0^1)}
| D^2u|^{\beta(d+1)}\,dx\,dt
\right)^ {\gamma/(\beta d+\beta)}
\mu^ {\gamma/(\beta' d+\beta')}
$$
\begin{equation}
                         \label{ppposc2}
+N\kappa^{-\alpha\gamma}\left(\dashint_{Q^+_{\kappa r}(x_0^1)}
 |D^2 u|^{d+1} \,dx\,dt\right)^{\gamma/(d+1)},
\end{equation}
where $N=N(\delta,d ,\beta )$ and $\beta'=\beta/(\beta-1)$.
\end{lemma}
\begin{proof}
 Introduce
$$
\bar{F}(u'')=\left\{
               \begin{array}{ll}
                 (F)_{Q^{+}_{R_0}(\tau,z)}(u''), & \hbox{if $\kappa r\geq R_0$;} \\
                 (F)_{Q^{+}_{\kappa r}(x_{0}^{1})}(u''), & \hbox{otherwise,}
               \end{array}
             \right.
$$
and
$$
h(t,x)=\sup_{u''\in \cS: |u''|=1}|F(u'',t,x)-\bar F(u'')|.
$$
Note that
$$
\partial_t u(t,x)+\bar F(D^2 u) =
\tilde f,
$$
where
$$
\tilde{f}(t,x)= f(t,x)+\bar F(D^2 u)-F(D^2 u,t,x).
$$
By Lemma \ref{lemma 3.21.3} and the triangle inequality,
$$
\dashint_{Q_r^+(x_0^1)}\dashint_{Q_r^+(x_0^1)} |D^2u(t,x)-D^2u(s,y)|^\gamma \,dx\, dt \,dy\,ds
$$
$$
\le N\kappa^{d+2}\left(\dashint_{Q^+_{\kappa r}(x_0^1)}|f|^{d+1} \,dx\,dt\right)^{\gamma/(d+1)}
+N\kappa^{d+2}J^{\gamma/(d+1)}
$$
\begin{equation}
               \label{3.21.9}
+ N\kappa^{-\alpha \gamma} \left(
\dashint_{Q^+_{\kappa r}(x_0^1)}
 |D^2 u|^{d+1} \,dx\,dt\right)^{\gamma/(d+1)},
\end{equation}
where $N=N(d,\delta)$,
$$
J=\dashint_{Q^+_{\kappa r}(x_0^1)}  | \bar F(D^2 u)-F(D^2 u,t,x)|^{d+1}
I_{Q^+_{R_0}(\tau,z)} \, dx\, dt
\leq J_1^{1/\beta}J_2^{1/\beta'},
$$
Here
$$
J_1 =\dashint_{Q^+_{\kappa r}(x_0^1)}  |D^2u|^{\beta(d+1)} \,dx\,dt,
$$
$$
J_2=\dashint_{Q^+_{\kappa r}(x_0^1)}  h^{\beta'
(d+1)}I_{Q^+_{R_0}(\tau,z)}\,dx\,dt
\leq N\dashint_{Q^+_{\kappa r}(x_0^1)}  hI_{Q^+_{R_0}(\tau,z)}\,dx\,dt.
$$
If $\kappa r < R_0$, we have
$$
J_2\leq N\dashint_{Q^+_{\kappa r}(x_0^1)} h(t,x) \,dx\,dt \leq N
\mu.
$$
If $\kappa r \geq R_0$, we have
$$
J_2 \le N(\kappa r)^{-d-2}\int_{Q^+_{R_0}(\tau,z)} h(t,x) \,dx\,dt
$$
$$
\le N(\kappa r)^{-d-2}R_0^{d+2}\dashint_{Q^+_{R_0}(\tau,z)} h(t,x) \,dx\,dt
\le N\mu.
$$
Therefore, in any case,
$$
J\leq
 N\left(\dashint_{Q^+_{\kappa r}(x_0^1)}
| D^2u(x)|^{\beta(d+1)}\,dx\,dt\right)^{1/\beta}\mu^{
1/\beta'}.
$$
Substituting the above inequality back into \eqref{3.21.9} yields \eqref{ppposc2}. The lemma is proved.
\end{proof}

The proof of Lemma \ref{pposc2} is
just a rather dull repetition of already given proofs of similar
facts. The following corollary is obtained in the same way as
similar assertions were obtained before.
\begin{corollary}
                         \label{3.16.1.2}
Let $p>d+1$ and $u\in \WO^{1,2}_{d+1}(\bR^{d+1}_+)$ be a solution to \eqref{vpara}
with $\lambda=0$ vanishing outside $Q_{R_0}^+(\tau,z)$, where $(\tau,z)\in \bR^{d+1}_+$. Then
there exist constants $\theta\in (0,1]$ and $N$ depending only on $p, d$, and
$\delta$, such that if Assumption \ref{assump1} is satisfied with this
$\theta$, then
$$
\|D^2 u\|_{L_p}+ \|\partial_{t}u\|_{L_p} \leq N\|f\|_{L_p}.
$$
\end{corollary}

Next we state the main result of this section,
which is deduced from Corollary \ref{3.16.1.2} by modifying
the proof  of Theorem  \ref{theorem 8.5.1}.
 By $\WO^{1,2}_{p}((T,\infty)\times \bR^d_+) $
we denote the set of functions of class
$W^{1,2}_{p}((T,\infty)\times \bR^d_+) $ with zero trace
on $x^{1}=0$.

\begin{theorem}
                                \label{thm7.3}
Let $p > d+1$ and $T\in [-\infty,\infty)$. There
exist constants $\theta=\theta(d,\delta,p)\in (0,1]$,
and $\lambda_0$ depending only on $d$, $p$, $\delta$
and $R_0$, such that if Assumption \ref{assump1}
holds with this $\theta$, then

(i) For any $\lambda\geq\lambda_0$ and $u\in
 \WO^{1,2}_{p}((T,\infty)\times \bR^d_+) $ satisfying
\eqref{vpara2}, we have
$$
\lambda\|u\|_{L_p((T,\infty)\times \bR^d_+)}+\|
\partial_{t}u\|_{L_p((T,\infty)\times \bR^d_+)}+
\|D^2 u\|_{L_p((T,\infty)\times \bR^d_+)}
$$
$$
 \leq N\|f\|_{L_p((T,\infty)\times \bR^d_+)},
$$
where $N=N(d,\delta,p)$.

(ii) For any $\lambda>0$, there exists a constant
$N=N(d,p, \delta ,R_0,\lambda)$ such that for
 any $u\in \WO^{1,2}_p((T,\infty)\times \bR^d_+) $
satisfying \eqref{vpara2}, we have
$$
\|u\|_{W^{1,2}_p((T,\infty)\times \bR^d_+)}
\le N\|f\|_{L_p((T,\infty)\times \bR^d_+)}.
$$

(iii) For any $\lambda> 0$ and  $f\in L_p((T,\infty)\times \bR^d_+)$, there exists a unique solution
$u\in \WO^{1,2}_p((T,\infty)\times \bR^d_+)$ of \eqref{vpara2}.
 \end{theorem}

We are now ready to prove Theorem \ref{mainthm1}.

\begin{proof}[Proof of Theorem \ref{mainthm1}]
The proof is similar to that of Theorem \ref{mainthm2}
in Section \ref{SecproofThm2} with some minor modifications.
As before, we first establish \eqref{eq16.01} as an a priori estimate and we may
assume  that $g\equiv 0$.

We will see again that to obtain the a priori estimate
we do not need  condition   (H$_3$).

Observe that Theorems \ref{7.13.2} and \ref{thm7.3}
with $\lambda=\lambda_{0}$ imply that
$$
\|\partial_{t}u\|_{L_p(\bR^{d+1}_{0})}+\|D^2u\|_{L_p(\bR^{d+1}_{0})}
\leq
 N(\|\partial_{t}u+F(D^{2}u) \|_{L_p(\bR^{d+1}_{0})}
$$
$$
+\| u \|_{L_p(\bR^{d+1}_{0})}
),\quad \forall u\in W^{1,2}_p(\bR^{d+1}_0) ,
$$
$$
\|\partial_{t}v\|_{L_p(\bR_{+}\times\bR^{d}_{+})}+
\|D^2v\|_{L_p(\bR_{+}\times\bR^{d}_{+})}\leq
 N(\|\partial_{t}v+F(D^{2}v) \|_{L_p(\bR_{+}\times\bR^{d}_{+})}
$$
\begin{equation}
                                                              \label{8.11.3}
+\| v \|_{L_p(\bR_{+}\times\bR^{d}_{+})} ) ,
\quad   \forall v\in  \WO^{1,2}_p (\bR_{+}\times\bR^{d}_{+}) ,
\end{equation}
where $N=N(d, p, \delta,R_{0})$ (provided that $\theta
=\theta(d,p,\delta)$ is chosen appropriately).

Now suppose that $u\in \WO^{1,2}_p(\cD_T)$ satisfies
\begin{equation}
                                                \label{7.14.2b}
\partial_t u+ F(D^{2}u, t,x)+G(D^{2}u,D u,u(t,x), t,x) =0
\end{equation}
in $\cD_T$. We extend $u$ and $G$ to be zero for $t>T$. It is easily
seen that the extended $u\in \WO^{1,2}_p(\cD_\infty)$ satisfies
\eqref{7.14.2b} in
$\cD_\infty$. Define
$$
f(t,x)=-G(D^2 u(t,x),Du(t,x),u(t,x),t,x).
$$

After that
by using the technique   based on flattening the boundary,
partitions of   unity, and interpolation inequalities
allowing one to estimate $Du$ through $D^{2}u$ and $u$  and
also using
\eqref{8.11.3} we obtain that
$$
\|\partial_{t}u\|_{L_p(\cD_{\infty})} + \|D^2 u\|_{L_p(\cD_{\infty})}
\le  N_1  (\| f \|_{L_p(\cD_{\infty})}+\| u \|_{L_p(\cD_{\infty})}),
$$
which is the same as
\begin{equation}
                        \label{8.11.4}
\|\partial_{t} u\|_{L_p(\cD_{T})} + \|D^2 u\|_{L_p(\cD_{T})}
\le N_{1} (\| f \|_{L_p(\cD_{T})}+\| u \|_{L_p(\cD_{T})}),
\end{equation}
provided that $\theta$ is sufficiently small
 depending only on $d$, $p$,
$\delta$, and the $C^{1,1}$ norm of $\partial D$.
Here $N_1$ depends only on $d$, $p$, $\delta$,  $R_0$,  and the  $C^{1,1}$
norm of $\partial \cD$.

It follows from the definition of $f$ and $({\rm H}_2)$ that, for any $s>0$,
$$
\|f\|_{L_p(\cD_T)}\le \chi(s)\|D^2
u\|_{L_p(\cD_T)}+\|\chi\|_{L_\infty}sT^{1/p}|\cD|^{1/p}
$$
\begin{equation}
                            \label{eq23.45p}
+
K(\|Du\|_{L_p(\cD_T)}+\|u\|_{L_p(\cD_T)})+\|\bar{G}\|_{L_p(\cD_T)}.
\end{equation}
Upon taking $s$ large such that $N_1\chi(s)\le 1/2$,
we get from \eqref{8.11.4}, \eqref{eq23.45p}
and the interpolation inequality that
\begin{equation}
                        \label{eq00.35}
\|u\|_{W^{1,2}_p(\cD_T)}
\le  N_2(\|u\|_{L_p(\cD_T)}+\|\bar G\|_{L_p(\cD_T)}+
\|\chi\|_{L_\infty}sT^{1/p}|\cD|^{1/p}),
\end{equation}
where $N_{2}$ is the same type of constant as $N_{1}$.

Next,  one can  estimate the $L_p(\cD_T)$ norm of $u$ by rewriting
\eqref{7.14.2b} similarly to \eqref{eq09.30}
as
$$
\partial_t u+Lu+b^iD_i u-cu=-G(D^{2}u,0,0,t,x)
$$
and using the parabolic Alexandrov estimates.   This will lead to an
a priori estimate \eqref{eq16.01} as in the proof of Theorem
 \ref{mainthm2} with $N$ depending also on $T$. To see that $N$ can be
chosen to be independent of $T$, we suppose without loss of generality
that
$\cD\subset B_{R/2}$,
where $R=4\text{diam}(\cD)$, and take the barrier
function $v_0$ defined on $\bR^d$ from Lemma~11.1.2 of \cite{Kr08}, which
satisfies in $B_R$,
$$
v_0>0,\quad Lv_0+b^{i}D_{i}v_0-cv_0\le -1.
$$
Denote $v=u/v_0$. Then $v\in \WO^{1,2}_p(\cD_T)$ satisfies
$$
\partial_t v+ Lv+\tilde b^{i}D_{i}v
-\tilde cv =-v_0^{-1}G(D^2(v_0v),0,0,t,x)
$$
in $\cD_T$, where
$$
\tilde b^i=b^i+2a^{ij}v_0^{-1}D_j v_0,\,\,\,
\tilde c=-v_0^{-1}\left( Lv_0+b^{i}D_{i}v_0-cv_0\right).
$$
It is easily seen that we can find constants $\tilde K>0$ and
$\nu>0$ depending only on
$d$,
$\delta$, $K$, and $R$, such that
$$
|\tilde b|\le \tilde K,\quad \nu\le \tilde c\le \tilde K.
$$
We then write $\tilde c=\hat c+\nu$ so that $\hat c\ge 0$.
As in the proof of Theorem \ref{3.22.1} (ii), it holds that
$$
\nu\|v\|_{L_p(\cD_T)}\le N(d,\delta,p)\|v_0^{-1}G(D^{2}(v_0 v),0,0,t,
x)\|_{L_p(\cD_T)},
$$
which gives
\begin{equation}
                                        \label{eq00.340}
\|u\|_{L_p(\cD_T)}\le N(d,\delta,p,R)\|G(D^{2}u,0,0,t, x)\|_{L_p(\cD_T)},
\end{equation}
owing to the properties of
$v_0$.
Combining \eqref{eq00.340} and \eqref{eq00.35}, we finish proving the a
priori estimate as in the proof of Theorem \ref{mainthm2}.

With the a priori estimate \eqref{eq16.01} in hand, the
existence and uniqueness are obtained by the same argument as at the end of Section
\ref{SecproofThm2}  relying on  condition  (H$_3$).
 The theorem is proved.
\end{proof}

\end{document}